\documentclass[12pt, reqno]{amsart}
\usepackage{amssymb,latexsym,amsmath,amsfonts,amsthm,amsxtra}
\usepackage{latexsym}
\usepackage[mathscr]{eucal}
\usepackage{bm}
\usepackage{rotating}
\usepackage{pgfplots}
\usepackage{hyperref}
\usepackage{emptypage}
\usepackage{enumerate}
\usepackage{enumitem}

\newcommand{\C}{\mathbb{C}}
\newcommand{\Ccc}{\mathbb{C}^3}
\newcommand{\Cn}{\mathbb{C}^n}
\newcommand{\D}{\mathbb{D}}

\newcommand{\Gg}{\mathbb{G}_2}
\newcommand{\Ggg}{\mathbb{G}_3}

\newcommand{\G}{\widetilde{\mathbb{G}}_3}

\newcommand{\gn}{\mathbb{G}_n}
\newcommand{\Gn}{\widetilde{\mathbb{G}}_n}
\newcommand{\gamn}{\Gamma_n}
\newcommand{\Gamn}{\widetilde{\Gamma}_n}

\newcommand{\vp}{\psi}
\newcommand{\M}{\mathcal{M}}

\newcommand{\be}{\beta}
\newcommand{\lm}{\lambda}

\newcommand{\q}{\quad}
\newcommand{\qq}{\qquad}
\newcommand{\n}{\lVert}

\newcommand{\lf}{\left(}
\newcommand{\ls}{\left\{}
\newcommand{\lt}{\left[}
\newcommand{\rf}{\right)}
\newcommand{\rs}{\right\}}
\newcommand{\rt}{\right]}
\newcommand{\df}{\dfrac}

\newcommand{\nj}{{n \choose j}}

\def ~{\hspace{1mm}}

        \def\textmatrix#1&#2\\#3&#4\\{\bigl({#1 \atop #3}\ {#2 \atop #4}\bigr)}
        \def\dispmatrix#1&#2\\#3&#4\\{\left({#1 \atop #3}\ {#2 \atop #4}\right)}

        \newtheorem{defn}{Definition}[section] 

        \newtheorem{thm}[defn]{Theorem}
        \newtheorem{rem}[defn]{Remark}

\begin{document}
\title[A Schwarz lemma for two domains and complex geometry]
{A Schwarz lemma for two families of domains and complex geometry}

\author[Sourav Pal]{Sourav Pal}
\address[Sourav Pal]{Mathematics Department, Indian Institute of Technology Bombay, Powai, Mumbai - 400076, India.} 
\email{sourav@math.iitb.ac.in}

\author[Samriddho Roy]{Samriddho Roy}
\address[Samriddho Roy]{Mathematics Department, Indian Institute of Technology Bombay, Powai, Mumbai - 400076, India.} \email{samriddhoroy@gmail.com}



\thanks{The first author is supported by the Seed Grant of IIT Bombay, the CPDA and the INSPIRE Faculty Award (Award No. DST/INSPIRE/04/2014/001462) of DST, India. The second author is supported by a Ph.D fellowship from the University Grand Commission of India.}

\begin{abstract}
We make sharp estimates to obtain a Schwarz type lemma for the symmetrized polydisc $\gn$ and for the extended symmetrized polydisc $\Gn$. We explicitly construct an interpolating function under certain condition. To do so, we followed the methods described in \cite{Young-LMS}. Also we find a few geometric interplay between the members of the family $\Gn$ and its closure $\widetilde{\Gamma}_n$.
\end{abstract}

\maketitle

\section{Introduction}

\noindent This article is a sequel of \cite{pal-roy 2}. Being motivated by the inspiring works due to Bharali, Costara, Edigarian, Kosinski, Nikolov, Zwonek \cite{bharali, costara1, edi-zwo, KZ, NN, zwonek4, zwonek5, NW, zwonek3} and few others (see references there in), we descend one more step into the depth of studying complex geometry and function theory of the symmetrized $n$-disk $\gn$ for $n\geq 3$. The \textit{symmetrized} $n$-\textit{disk} $\mathbb G_n$ or simply the \textit{symmetrized polydisc}, which consists of symmetric polynomials, is defined by
\begin{equation*}
\mathbb G_n =\left\{ \left(\sum_{1\leq i\leq n} z_i,\sum_{1\leq
i<j\leq n}z_iz_j,\dots,
\prod_{i=1}^n z_i \right): \,|z_i|< 1, i=1,\dots,n \right \}\,.
\end{equation*}
This domain arises in the famous $\mu$-synthesis problem, which is a part of the theory of robust control of systems comprising of interconnected electronic devices whose outputs
are linearly dependent on the inputs. Given a \textit{structure} $E$, which is a linear subspace of $\mathcal M_{m \times n}(\mathbb C)$, the space of all $m \times n$ matrices, the functional
\begin{align*}
& \mu_E(B)\\
& := (\text{inf} \{ \|X \|: X\in E \text{ and } (I-BX)
\text{ is singular } \})^{-1}, \; B\in \mathcal M_{ m \times
n}(\mathbb C),
\end{align*}
is called a \textit{a structured singular value}. If $m=n$ and if $E$ is the space of all scalar multiples of the identity matrix $I$, then $\mu_E(B)$ is equal to the spectral radius $r(B)$. Also if $E=\mathcal M_{m \times n}(\mathbb C)$, then $\mu_E (B)$ is precisely the operator norm $\|B\|$. Naturally if $E$ is any linear subspace of $\mathcal M_n(\mathbb C)$ that contains the identity matrix, then $r(B)\leq \mu_E(A) \leq \|B\|$. For the control-theory motivations behind $\mu_E$, we refer to the pioneering work of Doyle \cite{JCD}. The $\mu$-synthesis problem aims to find an analytic function $f$ from the open unit disk $\mathbb D$ of the complex plane $\mathbb C$ to $\mathcal M_{ m \times n}(\mathbb C)$ subject
to a finite number of interpolation conditions such that
$\mu_E(f(\lambda))<1$, for all $\lambda \in \mathbb D$. If  $E=\{
\lambda I:\lambda \in \mathbb C \} \subseteq \mathcal M_n (\mathbb
C)$, then $\mu_E (B)=r(B)<1$ if and only if
$\pi_n(\nu_1,\dots, \nu_n) \in \mathbb G_n$ (see
\cite{costara1}); here $\nu_1, \dots , \nu_n$ are
eigenvalues of $B$ and $\pi_n$ is the symmetrization map defined on $\mathbb C^n$ by
\[
\pi_n(z)=\left(s_1(z),\dots, s_{n-1}(z), p(z)\right),\; z=(z_1,\dots ,z_n)\,,
\]
where
\[
s_i(z)= \sum_{1\leq k_1 \leq k_2 \cdots \leq k_i \leq n}
z_{k_1}\cdots z_{k_i} \quad \text{ and }
p(z)=\prod_{i=1}^{n}z_i\,.
\]
It is merely mentioned that $\pi_n$ is a proper holomorphic map and $\pi_n(\D^n)=\mathbb G_n$, where $\mathbb D^n$ is the open polydisc defined by
\[
\mathbb D^n =\{ (z_1,\dots,z_n) : |z_i|<1, i=1,\dots,n \}.
\]
The \textit{closed symmetrized polydisc} $\Gamma_n$, which is the closure of $\gn$, is given by
\begin{align*}
\Gamma_n  : &=\left\{ \left(\sum_{1\leq i\leq n} z_i,\sum_{1\leq
i<j\leq n}z_iz_j,\dots, \prod_{i=1}^n z_i \right): \,|z_i|\leq 1,
i=1,\dots,n \right \} \\
&= \pi_n(\overline{\mathbb D^n}).
\end{align*}
The set $\Gamma_n$ is polynomially convex but not convex (see \cite{edi-zwo}). It is evident from the definition that $\mathbb G_1 =\mathbb D$ and below we provide an explicit form of $\Gg$ and $\Ggg$ for the convenience of the readers.
\begin{gather*}
\Gg = \ls (z_1 + z_2, z_1z_2) : z_1, z_2 \in \D \rs,\\
\Ggg = \ls (z_1 + z_2 + z_3,\; z_1z_2 + z_2z_3 + z_3z_1,\;
z_1z_2z_3) : z_1,z_2,z_3 \in \D \rs.
\end{gather*}
The symmetrized polydisc has attracted considerable attentions in past two decades because of its rich function theory \cite{ALY12, ALY-14, bharali 07, costara1, KZ, NTD, PZ}, complex geometry \cite{costara, edi-zwo, NN, zwonek4, zwonek5}, associated operator theory \cite{ay-jfa, tirtha-sourav, tirtha-sourav1, BSR, sourav, sourav3, pal-shalit}. An interested reader can also see the articles referred there.\\

The classical Schwarz lemma in one variable is stated in the following way.
\begin{thm}\label{classical-S}
    Let $f$ be an analytic function on $\mathbb{D}$ such that $|f(z)| \leq 1$, for all $z \in \mathbb{D}$ and $f(0)=0$. Then
    \begin{enumerate}
        \item[(a)] $|f(z)|\leq |z|$, for all $z \in \mathbb{D}$,
        \item[(b)] $|f'(0)|\leq 1$.
    \end{enumerate}
Moreover, if $|f'(0)| = 1$ or if $|f(z_0)|= |z_0|$ for some $z_0 \neq 0$, then there is a constant $c$ such that $|c| = 1$ and $f(w) = cw$ for all $w \in \mathbb{D}$.
\end{thm}

In \cite{pal-roy 2}, we obtained an analogue of the first part of Theorem \ref{classical-S} for the symmetrized polydisc. The main aim of this article is to continue the same program to find an analogous part-(b) for $\gn$ of the classical Schwarz lemma.\\

To study the complex geometry of $\gn$ (and $\Gamma_n$) more deeply and for proving a Schwarz lemma for $\gn$, we introduced a new family of domains in \cite{pal-roy 2}, which we named \textit{extended symmetrized polydisc} and defined as
\begin{align*}
\widetilde{\mathbb G}_n := \Bigg\{ (y_1,\dots,y_{n-1}, q)\in & \C^n :\: q \in \D, \:  y_j = \be_j + \bar \be_{n-j} q, \: \beta_j \in  \C \: \text{ and }\\
& \q |\beta_j|+ |\beta_{n-j}| < {n \choose j}\,,\q j=1,\dots, n-1 \Bigg\}.
\end{align*}
We called the closure of $\Gn$, the \textit{closed extended symmetrized polydisc} and denoted it by $\widetilde{\Gamma}_n$. We proved in \cite{pal-roy 2} that
\begin{align*}
\widetilde{\Gamma}_n := \Bigg\{ (y_1,\dots,y_{n-1}, q)\in & \C^n :\: q \in \D, \:  y_j = \be_j + \bar \be_{n-j} q, \: \beta_j \in  \C \: \text{ and }\\
& \q |\beta_j|+ |\beta_{n-j}| \leq {n \choose j}, \q j=1,\dots, n-1 \Bigg\}.
\end{align*}
The purpose of introducing the family $\Gn$ was to make a few sharp estimates which provides a Schwarz lemma for $\gn$, \cite{pal-roy 2}. Also we obtained a variety of characterizations for the points in $\gn$ and $\Gamma_n$ via a similar set of characterizations for $\Gn$ and $\widetilde{\Gamma}_n$ respectively, \cite{pal-roy 2}.\\

In \cite{costara1}, Costara showed that 
\[
\gn = \big\{ (s_1,\dots, s_{n-1}, p) : p\in \D,\;  s_j=\be_j + \bar\be_{n-j}p \; \& \; (\be_1, \dots, \be_{n-1}) \in \mathbb G_{n-1} \big\}
\]
and 
\[
\gamn = \big\{ (s_1,\dots, s_{n-1}, p) : p\in \overline\D,\; s_j=\be_j + \bar\be_{n-j}p \; \& \;(\be_1, \dots, \be_{n-1}) \in \Gamma_{n-1} \big\}.
\]
It is obvious that if $\lf \be_1,\dots, \be_{n-1} \rf \in \mathbb G_{n-1}$, then $|\be_j| + |\be_{n-j}| < \nj$. Therefore, it follows that $\gn \subseteq \Gn$. In fact, $\mathbb G_2 = \widetilde{\mathbb G}_2$ but $\gn \subsetneq \Gn$ for $n \geq 3$ (see \cite{pal-roy 2}, Lemma 3.0.2). \\

We introduced $n-1$ fractional linear transformations $\Phi_1,\dots ,\Phi_{n-1}$ and with their help we made some sharp estimates to find necessary conditions for the existence of an interpolating function from $\mathbb D$ to $\Gn$ and since $\gn \subseteq \Gn$, the estimates became necessary for a Schwarz lemma for $\gn$ (see \cite{pal-roy 2}). Since the maximum modulus of each co-ordinate of a point in $\Gn$ does not exceed that of a point in $\gn$, these estimates are sharp for $\gn$ too. Moreover, the functions $\Phi_1,\dots , \Phi_{n-1}$ are specially designed for $\Gn$ and they characterize the points in $\Gn$ and $\widetilde{\Gamma}_n$. \\

In this article, we first prove an analogue of part-(b) of Theorem \ref{classical-S} for $\Gn$, which is Theorem \ref{Schwarz-lemma-Gn-2} and it is one of the main results of this paper. As a consequence the desired Schwarz lemma for $\gn$ (Theorem \ref{Schwarz-lemma-gn-2}) follows. We also show in Theorem \ref{Schwarz-lemma-Gn-2} that under certain condition, the achieved estimates are sufficient for the existence of an interpolating function from $\mathbb D$ to $\Gn$. In Section 4, we explicitly construct such an interpolating function. Section 5 deals with some geometric interplay between the members of $\Gn$ and $\widetilde{\Gamma_n}$. In Section 2, we accumulate few results from the literature which are used in the subsequent sections.\\

\noindent \textbf{Note.} The main idea and applied techniques to the results of Sections $3$ and $4$ of this article are borrowed from the paper \cite{Young-LMS}, where analogous results for the tetrablock $\mathbb E$ are achieved. The primary reason for which the techniques of \cite{Young-LMS} are applicable here is that $\widetilde{\mathbb G_3}$ is linearly isomorphic to $\mathbb E$.

\section{Background materials and preparatory results}
\noindent We begin with a set of $(n-1)$ fractional linear transformations $\Phi_1, \dots, \Phi_{n-1}$ which we introduced in \cite{pal-roy 2} to characterize the points in the extended symmetrized polydisc $\Gn$.
\begin{defn}
	For $z \in \C$, $y=(y_1,\dots,y_{n-1},q) \in \Cn $ and for any $j\in \left\{1,\dots,n-1\right\}$,
	let us define
	\begin{equation} \label{defn-p}
	\Phi_j(z,y) =
	\begin{cases}
	\dfrac{{n \choose j}qz-y_j}{y_{n-j}z-{n \choose j}} & \q \text{ if } y_{n-j}z\neq {n \choose j} \text{ and } y_j y_{n-j}\neq {n \choose j}^2 q \\
	\\
	\dfrac{y_j}{{n \choose j}} & \q \text{ if } y_j y_{n-j} = {n \choose j}^2 q \, .
	\end{cases}
	\end{equation}
\end{defn}
It was shown in \cite{pal-roy 2} that if $|y_{n-j}| < {n \choose j}$,  then
\begin{equation} \label{formula-D}
\n \Phi_j(.,y)\n_{H^{\infty}}=
\dfrac{{n \choose j} \left|y_j - \bar y_{n-j} q \right| + \left|y_j y_{n-j} - {n \choose j}^2 q \right|}{{n \choose j}^2 - |y_{n-j}|^2 } .
\end{equation}
Clearly $|y_{n-j}| < {n \choose j}$, for any point $y \in \Gn$. Thus equation $\eqref{formula-D}$ holds for any $y \in \Gn$.\\

For $ n\geq 3$, we introduced in \cite{pal-roy 2} the following subset $\mathcal J_n$ of $\Gn$ as follows:
\begin{equation}\label{eqn:21}
\mathcal J_n = 
\begin{cases}
\mathcal J_n^{odd} & \text{ if } n \text{ is odd}\\
\mathcal J_n^{even} & \text{ if } n \text{ is even},
\end{cases}
\end{equation}
where
\begin{align*}
\mathcal J_n^{odd} = \Big\{(y_1, \dots, y_{n-1}, y_n) \in \Gn : y_j = \dfrac{\nj}{n} y_1,\; & y_{n-j} = \dfrac{\nj}{n}  y_{n-1}\, , \\ 
&\text{for } j=2,\dots, \left[\frac{n}{2} \right] \Big\}
\end{align*}
and 
\begin{align*}
\mathcal J_n^{even} = \Big\{(y_1, \dots, y_{n-1}, y_n) \in \Gn : & y_{[\frac{n}{2} ]}= {n \choose [\frac{n}{2} ]} \df{y_1 + y_{n-1}}{2n},\;  y_j = \dfrac{\nj y_1}{n}, \\ 
y_{n-j} &= \dfrac{\nj y_{n-1}}{n}, \text{ for }j=2,\dots, \left[\frac{n}{2} \right]-1  \Big\}.
\end{align*}

The following theorem provides a few characterizations for the points in $\Gn$.

\begin{thm}[\cite{pal-roy 2}, Theorem 3.1.4]\label{char G 3}
	For a point $y =(y_1,\dots, y_{n-1},q) \in \Cn$, the following are equivalent:
	\begin{itemize}[leftmargin=*]
		\item[$(1)$] $y \in \Gn$;\\
		\item[$(2)$] ${n \choose j} - y_j z - y_{n-j}w + { n \choose j} qzw \neq 0$, for all $z,w \in \overline\D$ and for all $j = 1, \dots, \left[\frac{n}{2}\right]$;\\
		\item[$(3)$] for all $j = 1, \dots, \left[\frac{n}{2}\right]$ either
		\[
		{n \choose j}\left|y_j - \bar y_{n-j} q\right| + \left|y_j y_{n-j} - {n \choose j}^2 q \right| < {n \choose j}^2 -|y_{n-j}|^2
		\]
		or 
		\[
		{n \choose j}\left|y_{n-j} - \bar y_j q\right| + \left|y_j y_{n-j} - {n \choose j}^2 q \right| < {n \choose j}^2 -|y_j|^2 ;
		\]
		\item[$(4)$] $\left| y_{n-j} - \bar y_j q \right| + \left| y_j -\bar y_{n-j} q\right| < {n \choose j} (1 - |q|^2)$ for all $j = 1, \dots, \left[\frac{n}{2}\right]$;\\
		\item[$(5)$] there exist $\left[\frac{n}{2}\right]$ number of $2 \times 2$ matrices $B_1,\dots, B_{\left[\frac{n}{2}\right]}$ such that $\n B_j \n < 1$, $y_j = {n \choose j}[B_j]_{11} $, $y_{n-j} = {n \choose j}[B_j]_{22}$ for all $j = 1, \dots, \left[\frac{n}{2}\right]$ and 
		\[\det B_1= \det B_2 = \cdots = \det B_{\left[\frac{n}{2}\right]}= q \; .\]
	\end{itemize}
\end{thm}

In \cite{pal-roy 2}, we obtained several equivalent necessary conditions which established a Schwarz type lemma for $\Gn$. Here we mention a few of them.

\begin{thm}[\cite{pal-roy 2}, Theorem 5.2.1]\label{Schwarz Gn}
	Let $\lm_0 \in \D \; \backslash \; \{0\}$ and let $ y^0= (y_1^0,\dots,y_{n-1}^0,q^0) \in \Gn$. Then in the set of following conditions, $(1)$ implies $(2)$ and $(3)$.
	\item[$(1)$]
	There exists an analytic function $\vp  :  \D \rightarrow  \Gn $ such that  $\; \vp(0) = (0,\dots,0) $ and $\; \vp(\lm_0) = y^0 $.
	\item[$(2)$]
	\[ 
	\max_{1\leq j \leq n-1} \ls \n \Phi_j(.,y^0)\n_{H^{\infty}} \rs \leq |\lm_0| \, .
	\]
	
%
%
	\item[$(3)$]  There exist $\left[\frac{n}{2}\right]$ number of functions $F_1, F_2, \dots F_{\left[\frac{n}{2}\right]}$ in the Schur class such that  
	$F_j(0) =
	\begin{bmatrix}
	0 & * \\
	0 & 0
	\end{bmatrix} ,\:$  
	and $\; F_j(\lm_0) = B_j$, for $j = 1, \dots, \left[\frac{n}{2}\right]$, where $\det B_1= \cdots = \det B_{[\frac{n}{2}]}= q^0$,
	$y_j^0 = \nj [B_j]_{11}$ and $y_{n-j}^0 = \nj [B_j]_{22}$.\\

Furthermore, if $ y^0 \in \mathcal J_n $ then all the conditions $(1)-(3)$ are equivalent.	
\end{thm}

The following result is known as Parrott's Theorem and it will be used in sequel. One can see Theorem 12.22 in \cite{young-2} for a proof to this result.

\begin{thm}[\cite{parrott}, Theorem 1]\label{parrot}
	Let $H_i, K_i$ are Hilbert spaces for $i=1,2$, and let 
	$$ \begin{bmatrix}
	Q\\
	\\
	S
	\end{bmatrix}
	:H_2 \longrightarrow K_1 \oplus K_2, \qq 
	\begin{bmatrix}
	R & S
	\end{bmatrix}
	: H_1 \oplus H_2 \longrightarrow K_2$$
	be contractions. Then there exists $P \in \mathcal{L}(H_1,K_1)$ such that 
	$$ \begin{bmatrix}
	P & Q \\
	R & S
	\end{bmatrix}
	: H_1 \oplus H_2 \longrightarrow K_1 \oplus K_2$$
	is a contraction.
\end{thm}

\section{A Schwarz lemma for $\Gn$ and $\gn$}
\noindent Let $B_1,\dots, B_k$ be $2 \times 2$ contractive matrices such that $\det B_1= \det B_2 = \cdots = \det B_k$. We define two functions $\pi_{2k+ 1}$ and $\pi_{2k}$ in the following way:
\begin{align*}
&\pi_{2k+1} \lf B_1,\dots, B_k \rf \\
& = \lf {n \choose 1} [B_1]_{11}, \dots , {n \choose k} [B_k]_{11}, {n \choose k} [B_k]_{22}, \dots, {n \choose 1} [B_1]_{22} , \det B_1 \rf \\
\end{align*}
and\\
$\pi_{2k} \lf B_1,\dots, B_k \rf $
\begin{align*}
=\Bigg( {n \choose 1} [B_1]_{11}, \dots , {n \choose k-1} [B_{k-1}]_{11},& {n \choose k} \df{\lf [B_k]_{11} + [B_k]_{22} \rf}{2},\\ {n \choose k-1}& [B_{k-1}]_{22}, \dots, {n \choose 1} [B_1]_{22}, \det B_1  \Bigg).
\end{align*}
Then by Theorem \ref{char G 3}, we have 
\[\pi_{2k} \lf B_1,\dots, B_k \rf \in \widetilde{\mathbb G}_{2k} \q \text{ and } \q \pi_{2k+1} \lf B_1,\dots, B_k \rf \in \widetilde{\mathbb G}_{2k+1}.\]
 
\noindent For $ n\geq 3$, let $\mathcal K_n$ be the following subset of $\C^n$ :
\[
\mathcal K_n = 
\begin{cases}
\mathcal K_n^{odd} & \text{ if } n \text{ is odd}\\
\mathcal K_n^{even} & \text{ if } n \text{ is even},
\end{cases}
\]
where
\begin{align*}
	\mathcal K_n^{odd} = \Big\{& \lf y_1, \dots, y_{n-1}, y_n \rf \in \C^n : y_j = \dfrac{\nj}{n} y_1,\; y_{n-j} = \dfrac{\nj}{n}  y_{n-1}\,, \\ 
	& \q \text{for } j=2,\dots, \left[\frac{n}{2} \right]  \& \, \max \Big\{ \frac{|y_1|}{n}, \frac{|y_{n-1}|}{n} \Big\} + |y_n| \leq 1\Big\}
\end{align*}
and 
\begin{align*}
& \mathcal K_n^{even} = \Big\{ \lf y_1, \dots, y_{n-1}, y_n \rf \in \C^n : y_{[\frac{n}{2} ]}= {n \choose [\frac{n}{2} ]} \df{y_1 + y_{n-1}}{2n},  y_j = \dfrac{\nj y_1}{n}, \\& y_{n-j} = \dfrac{\nj y_{n-1}}{n},  \text{ for }j=2,\dots, \left[\frac{n}{2} \right]-1 \;\& \; \max \{ \frac{|y_1|}{n}, \frac{|y_{n-1}|}{n} \} + |y_n| \leq 1  \Big\}.
\end{align*}

For any $Z \in \C^{2 \times 2}$ with $\n Z \n < 1$, let $D_Z = (1-Z^*Z)^{\frac{1}{2}}$. We denote the unit ball of $\C^{2 \times 2}$ by $R_I(2,2)$. Consider the following function: 
$$ \M_Z(X) = -Z + D_{Z^*}X(1-Z^*X)^{-1}D_Z \q \text{for } X \in R_I(2,2).$$ 
The function $\M_Z$ is a matrix M\"{o}bius transformation that maps $Z$ to $0$. The transformation $\M_Z$ is an automorphism of $R_I(2,2)$, and $(\M_Z)^{-1}=\M_{-Z}$.\\

We now present a Schwarz type lemma for $\Gn$.


\begin{thm}\label{Schwarz-lemma-Gn-2}
	Let $x=(x_1,\dots,x_n)	\in \C^n$ and there exists an analytic map $\vp : \D \longrightarrow \Gn$ such that $\vp(0)=(0,\dots,0)$ and $\vp'(0)=x$. Then
	\begin{equation}\label{condition-yn}
		\max_{1\leq j \leq n-1}\left\{ \df{|x_j|}{\nj}\right\} + |x_n| \leq 1. 
	\end{equation}	
	The converse holds if $x \in \mathcal K_n$.  
\end{thm}

\begin{proof}

We already know that $\widetilde{\mathbb G_2}=\mathbb G_2$ and this theorem was proved by Agler and Young for $\mathbb G_2$ (see Theorem 1.1 in \cite{AY-BLMS}). For this reason we shall consider $n\geq 3$ when proving the converse part of this theorem.\\

Let $\vp : \D \longrightarrow \Gn$ is an analytic map such that $\vp(0)=(0,\dots,0)$ and $\vp'(0)=x$. Write $\vp = (\vp_1,\dots,\vp_n)$. Then, by Theorem \ref{Schwarz Gn}, for each $\lm \in \D \setminus \{0\}$  we have 
	\[
	\max_{1\leq j \leq n-1} \ls \n \Phi_j(.,\vp(\lm)) \n_{H^{\infty}} \rs \leq |\lm| , 
	\]
	which is same as saying
	\begin{align}\label{max 1}
	\nonumber	\max_{1\leq j \leq n-1} &  \Bigg\{ \dfrac{{n \choose j} \left|\vp_j(\lm)- \bar{\vp}_{n-j}(\lm) \vp_n(\lm) \right| + \left|\vp_j(\lm) \vp_{n-j}(\lm) - {n \choose j}^2 \vp_n(\lm) \right|}{{n \choose j}^2 - |\vp_{n-j}(\lm)|^2 }  \Bigg\} \\
		 \leq & |\lm|,
	\end{align}
	for each $\lm \in \D \setminus \{0\}$. Note that, using L-Hospital's rule and the fact that $\vp(0)=(\vp_1(0),\dots,\vp_n(0))=(0,\dots,0)$, we have
	\begingroup
	\allowdisplaybreaks
	\begin{align*}
		&\lim\limits_{\lm \rightarrow 0} \dfrac{{n \choose j} \left|\vp_j(\lm)- \bar{\vp}_{n-j}(\lm) \vp_n(\lm) \right| }{|\lm| \lf {n \choose j}^2 - |\vp_{n-j}(\lm)|^2 \rf }\\
		=& \nj \left| \lim\limits_{\lm \rightarrow 0} \df{\vp_j(\lm)- \bar{\vp}_{n-j}(\lm) \vp_n(\lm)}{\lm} \right| \lim\limits_{\lm \rightarrow 0}\df{1}{ \lf {n \choose j}^2 - |\vp_{n-j}(\lm)|^2 \rf }
		 = \df{|\vp_j'(0)|}{\nj}
	\end{align*}
	\endgroup
	and
	\begingroup
	\allowdisplaybreaks
	\begin{align*}
		& \lim\limits_{\lm \rightarrow 0}\df{|\vp_j(\lm)\vp_{n-j}(\lm) - \nj^2 \vp_n(\lm)|}{|\lm| \lf {n \choose j}^2 - |\vp_{n-j}(\lm)|^2 \rf }
		 = \df{1}{\nj^2} \left|-\nj^2 \vp_n'(0) \right| = |\vp_n'(0)|.
	\end{align*}
	\endgroup
	So 
	$$  \lim\limits_{\lm \rightarrow 0} \df{\nj \left|(\vp_j - \bar{\vp}_{n-j}\vp_n)(\lm) \right| + \left|(\vp_j\vp_{n-j} -\nj^2 \vp_n)(\lm) \right| }{ |\lm| \lf {n \choose j}^2 - |\vp_{n-j}(\lm)|^2 \rf } = \df{|\vp_j'(0)|}{\nj} + |\vp_n'(0)|. $$
	Since inequality $\eqref{max 1}$ is true for all $\lm \in \D \setminus \{0\}$, by dividing both side of $\eqref{max 1}$ by $|\lm|$ and letting $\lm \longrightarrow 0$, we have 
	$$ \max_{1\leq j \leq n-1}\left\{ \df{|\vp_j'(0)|}{\nj} + |\vp_n'(0)| \right\} \leq 1. $$
	Since $\vp'(0)=x$, that is, $(\vp_1'(0), \dots, \vp_n'(0))= (x_1,\dots,x_n)$, we have
	\[ \max_{1\leq j \leq n-1}\left\{ \df{|x_j|}{\nj}\right\} + |x_n| \leq 1 .
	\]

\noindent We divide the converse part into two cases, $n=3$ and $n>3$.\\
	
\noindent \textbf{Case-I.} Suppose $n =3$ and condition $\eqref{condition-yn}$ holds for $x=(x_1,x_2,x_3) \in \mathcal K_3$, that is,
	\begin{equation}\label{condition-y}
		\max\Bigg\{ \df{|x_1|}{3}, \df{|x_2|}{3} \Bigg\} + |x_3| \leq 1. 
	\end{equation}
We show that there exists an analytic map $\vp : \D \longrightarrow \G$ such that $\vp(0)=(0,0,0)$ and $\vp'(0)=(x_1,x_2,x_3)$. We shall follow similar technique as in Theorem 2.1 of \cite{Young-LMS} to construct such a function. We first assume that $|x_2| \leq |x_1|$. So, if $x_1=0$, then we have $x_2=0$ and hence from the inequality $\eqref{condition-y}$, we obtain $|x_3|\leq 1$. Now consider the function $\vp(\lm)=(0,0,\lm x_3)$. Clearly $\vp$ is analytic and satisfies 
	\[
	\vp(0)= (0,0,0) \q \text{and} \q \vp'(0)= \vp'(\lm)|_{\lm =0} = (0,0,x_3)=x.
	\]
Next, assume $x_1 \neq 0$. According to Theorem \ref{char G 3}, for any $W\in \mathcal S_{2 \times 2}$ the function $\vp =\pi_3 \circ W = (3W_{11},3W_{22},W_{11}W_{22}-W_{12}W_{21})$ is an analytic map from $\D$ to $\G$. Hence it is enough to show the existence of a function $W = [W_{ij}] \in \mathcal S_{2 \times 2}$ such that $(\pi_3 \circ W)(0) = (0,0,0)$ and $(\pi_3 \circ W)'(0)=x$. Suppose $W$ is a $2 \times 2$ matrix valued function such that
	\begin{equation}\label{F(zeta)}
		W(0) = 
		\begin{bmatrix}
			0 & \sigma \\
			0 & 0
		\end{bmatrix},
	\end{equation}
	where $\sigma \in \D$ will be chosen later. Then, $\vp(0)=(\pi_3 \circ W)(0) = (0,0,0)$ and 
	\begingroup
	\allowdisplaybreaks
	\begin{align*}
		\vp'(0) 
		&= (3W_{11},3W_{22},W_{11}W_{22}-W_{12}W_{21})'(0) 
		 = (3W_{11}',3W_{22}', -\sigma W_{21}')(0).
	\end{align*}
	\endgroup
	Accordingly, $\vp'(0) = x$ if and only if  
	\begin{equation}\label{F'(zeta)}
		W'(0) = 
		\begin{bmatrix}
			x_1/3 & * \\
			-x_3/\sigma & x_2/3
		\end{bmatrix}.
	\end{equation}
	We shall find a function $W \in \mathcal S_{2 \times 2}$ that satisfies equations $\eqref{F(zeta)}$ and $\eqref{F'(zeta)}$.\\
	
	For any $W \in \mathcal S_{2 \times 2}$, we have 
	\[
	\lf \M_Z \circ W \rf (\lm) = -Z + D_{Z^*}W(\lm) \left( 1-Z^*W(\lm) \right)^{-1}D_Z \q \text{for } \lm\in \D
	\]
	and
	\begingroup
	\allowdisplaybreaks
	\begin{align}\label{M-Z'}
		(\M_Z \circ W)' 
		\nonumber &= D_{Z^*}[1 + W(1-Z^*W)^{-1}Z^*]W'(1-Z^*W)^{-1}D_Z\\
		&= D_{Z^*}(1-WZ^*)^{-1}W'(1-Z^*W)^{-1}D_Z .
	\end{align}
	\endgroup
	For a fixed $\sigma \in \D$, let
	\begin{equation}\label{Z-zeta}
		Z = 
		\begin{bmatrix}
			0 & \sigma \\
			0 & 0
		\end{bmatrix}.
	\end{equation}
	Then $\n Z \n <1$,
	\begin{equation*}
		D_Z = (1-Z^*Z)^{\frac{1}{2}} 
		= \begin{bmatrix}
		1 & 0 \\
		0 & (1-|\sigma|^2)^{\frac{1}{2}}
		\end{bmatrix}
		\text{ and }
		 D_{Z^*} 
		 = \begin{bmatrix}
		 	(1-|\sigma|^2)^{\frac{1}{2}} & 0 \\
		 	0 & 1
		 \end{bmatrix}.
	\end{equation*}
	Now if $W \in \mathcal S_{2 \times 2}$ satisfies $\eqref{F(zeta)}$ and $\eqref{F'(zeta)}$, then $W(0)=Z$ and hence we have 
	\begingroup
	\allowdisplaybreaks
	\begin{align}\label{M_F'}
		\nonumber (\M_Z \circ W)'(0)& = D_{Z^*}(1-ZZ^*)^{-1}W'(0)(1-Z^*Z)^{-1}D_Z \\
		&=\begin{bmatrix}
			\df{x_1}{3(1-|\sigma|^2)^{\frac{1}{2}}} & \df{W_{12}'(0)}{1-|\sigma|^2} \\
			\\
			\df{-x_3}{\sigma} & \df{x_2}{3(1-|\sigma|^2)^{\frac{1}{2}}}
		\end{bmatrix}.
	\end{align}
	\endgroup
	Note that if $W \in \mathcal S_{2 \times 2}$ then by $\eqref{F(zeta)}$ and $\eqref{Z-zeta}$, the map
$
	\M_Z \circ W : \D \longrightarrow R_I(2,2)
$
	satisfies the following condition
	$$(\M_Z \circ W)(0)=\M_Z(W(0))=\M_Z(Z)=0 .$$
	Therefore, by Schwarz lemma for $R_I(2,2)$, we have $\n (\M_Z \circ W)'(0) \n < 1$. Thus, if there exists a function $W$ in $\mathcal S_{2 \times 2}$ which satisfies $\eqref{F(zeta)}$ and $\eqref{F'(zeta)}$, then the matrix in the right hand side of the equation $\eqref{M_F'}$ must be a strict contraction. \\	
	Now choose $\sigma = \sqrt{1-\df{|x_1|}{3}}$. By \eqref{condition-y}, $|x_3|< 1$ hence $\sigma \in \D$. For a $\rho \in \C$, define a matrix $B_{\rho}$ by
	\begin{equation}\label{Y-zeta}
		B_{\rho}=\begin{bmatrix}
			\df{x_1}{\sqrt{3|x_1|}} & \rho \\
			\\
			\df{-x_3}{\sqrt{1-\df{|x_1|}{3}}} & \df{x_2}{\sqrt{3|x_1|}}
		\end{bmatrix}.  
	\end{equation}
	For a fixed $\rho$ (which is to be determined), define a function
	\begin{equation}\label{func V}
		V_{\rho}(\lm)= \lm B_{\rho}\,, \quad \lm \in \D.
	\end{equation} 
	Then,
	$V_{\rho}(0)=\begin{bmatrix}
	0 & 0 \\
	0 & 0
	\end{bmatrix}$ and $V_{\rho}'(0)=B_{\rho}$. We define $B_{\rho}$ in such a fashion because, for a suitable choice of $\rho$, the matrix $V_{\rho}'(0)$ is analogous to the matrix in equation $\eqref{M_F'}$ with above choice of $\sigma$. Since we have assumed $|x_2| \leq |x_1|$ and since condition $\eqref{condition-y}$ holds, we have that $|x_1|/3 + |x_3| \leq 1$. Thus, the norm of first column of $B_{\rho}$ is equal to
	\begin{align*} 
	\Bigg( \df{|x_1|}{3} + \df{|x_3|^2}{(1-|x_1|/3)} \Bigg)^{\frac{1}{2}} \leq 
	\Bigg( \df{|x_1|}{3} + \df{|x_3|(1-|x_1|/3)}{(1-|x_1|/3)} \Bigg)^{\frac{1}{2}} 
	&=\Bigg( \df{|x_1|}{3} + |x_3| \Bigg)^{\frac{1}{2}} \\
	&\leq 1 
	\end{align*}
	and also the norm of the second row of $B_{\rho}$ is
	\[
	\lf \df{|x_3|^2}{(1-|x_1|/3)} + \df{|x_2|^2}{3 |x_1|} \rf^{\frac{1}{2}} \leq \df{|x_3|^2}{(1-|x_1|/3)} + \df{|x_1|^2}{3 |x_1|} \leq 1 \: \: (\text{since } |x_2| \leq |x_1|).
	\] 
	Then, by Theorem \ref{parrot}, there exists a $\rho \in \C$ such that $\n B_{\rho} \n \leq 1$. Consequently, there exists a $\rho \in \C$ such that $V_{\rho} \in \mathcal S_{2 \times 2}$. A choice of such $\rho$ is 
	\begin{equation}\label{xi}
		\rho = \df{x_1x_2\bar{x}_3\sqrt{3-|x_1|}}{\sqrt{3}|x_1|(3-|x_1|-3|x_3|^2)}\;.
	\end{equation} 
	Now we define a function $W = \M_{-Z} \circ V_{\rho}$, where $Z$ is as in $\eqref{Z-zeta}$. Then $W \in \mathcal S_{2 \times 2}$ and $W(0)= \M_{-Z}(0)=Z$. So $W$ satisfies equation $\eqref{F(zeta)}$. Since $V_{\rho}(0)$ is the zero matrix and $V_{\rho}'(0)=B_{\rho}$, by equation $\eqref{M-Z'}$ we have
	\begingroup
	\allowdisplaybreaks 
	\begin{align*}
		W'(0)
		&= D_{Z^*}(1+V_{\rho}(0)Z^*)^{-1}V_{\rho}'(0)(1+Z^*V_{\rho}(0))^{-1}D_Z  \\
		&= D_{Z^*}V_{\rho}'(0)D_Z = \begin{bmatrix}
			x_1/3 & \rho|x_1|/3 \\
			-x_3/\sigma & x_2/3
		\end{bmatrix}.
	\end{align*}
	\endgroup
	Hence the function $W$ also satisfies equation $\eqref{F'(zeta)}$. Thus, there exists a function $W \in \mathcal S_{2 \times 2}$ that satisfies $\eqref{F(zeta)}$ and $\eqref{F'(zeta)}$. The case $|x_1| \leq |x_2|$ can be dealt in similar way. Hence condition \eqref{condition-y} is sufficient for the existence of an analytic map $\vp : \D \longrightarrow \G$ such that $\vp(0)=(0,\dots,0)$ and $\vp'(0)=(x_1,x_2,x_3)$.\\
	
\noindent \textbf{Case-II.} Let $n>3$. First assume that $n$ is odd. Suppose $x = (x_1, \dots, x_n) \in \mathcal K_n$. Then for all $j = 2, \dots, \Big[\dfrac{n}{2} \Big]$, we have $ x_j = \dfrac{\nj x_1}{n}$  and $x_{n-j} = \dfrac{\nj x_{n-1}}{n}$.
Therefore, condition \eqref{condition-yn} reduces to 
	\begin{equation*}
	\max\ls \df{|x_1|}{n}, \df{|x_{n-1}|}{n} \rs + |x_n| \leq 1 \q \Leftrightarrow \q  \max\ls \df{\left| \frac{3 x_1}{n} \right|}{3}, \df{\left| \frac{3 x_{n-1}}{n} \right|}{3} \rs + |x_n| \leq 1.
	\end{equation*}
	Hence by Case-I, there exists a function $\widehat W \in \mathcal S_{2 \times 2}$ such that
	\begin{equation}\label{widehat W}
	\widehat W(0) = 
	\begin{bmatrix}
	0 & \sigma \\
	0 & 0
	\end{bmatrix}
	 \text{ and } 
	 \widehat W'(0) = 
	 \begin{bmatrix}
	   x_1/n & * \\
	   -x_n/\sigma & x_{n-1}/n
	 \end{bmatrix}.	  
	\end{equation}
	Now consider $\Big[ \dfrac{n}{2} \Big]$ number of $2 \times 2$ matrix valued functions $W_1, \dots, W_{[\frac{n}{2}]}$, where $W_j (\lm) = \widehat W(\lm)$ for each $j= 1,\dots, \Big[ \dfrac{n}{2} \Big]$.  Since $n$ is odd, $2[\frac{n}{2}] +1 = n$. So, we have
	\begin{align*}
	 &\big( \pi_{2[\frac{n}{2}] +1} \circ ( W_1, \dots, W_{[\frac{n}{2}]} ) \big)'(0)\\
	  = & \Big( x_1, \dots, \df{{n \choose [\frac{n}{2}]}}{n} x_1, \df{{n \choose [\frac{n}{2}]}}{n} x_{n-1}, \dots, x_{n-1}, x_n \Big) 
	 = x 
	\end{align*}
	and 
	\begin{align*}
	\lf \pi_n \circ ( W_1, \dots, W_{[\frac{n}{2}]} ) \rf(0) = \pi_{2[\frac{n}{2}] +1} \lf \widehat W(0), \dots, \widehat W(0) \rf = (0,\dots,0).
	\end{align*}
	Note that, each $W_j \in \mathcal S_{2 \times 2}$ and $\det W_i=\det W_j$ for each $i,j$. Clearly, the function $\vp = \pi_n \circ \lf W_1, \dots, W_{[\frac{n}{2}]} \rf$ is analytic which maps $\D$ into $\Gn$.\\
	
Now suppose $n$ is even. In this case, $x = (x_1, \dots, x_n) \in \mathcal K_n^{even}$. Then, $x_j = \dfrac{\nj x_1}{n} $, $x_{n-j} = \dfrac{\nj x_{n-1}}{n} $ for all $j = 2, \dots, [\frac{n}{2}]-1$ and $x_{[\frac{n}{2}]} = {n \choose [\frac{n}{2} ]} \df{x_1 + x_{n-1}}{2n}$. Hence condition \eqref{condition-yn} is reduced to 
	\begin{equation*}
	\max\Bigg\{ \df{|x_1|}{n}, \df{|x_{n-1}|}{n}, \df{|x_1|+|x_{n-1}|}{2n} \Bigg\} + |x_n| \leq 1,
	\end{equation*}
	which is same as
	\begin{equation*}
	\max\Bigg\{ \df{|x_1|}{n}, \df{|x_{n-1}|}{n} \Bigg\} + |x_n| \leq 1. 
	\end{equation*}
So in a similar fashion as if $n$ is odd, there exists a function $\widehat W \in \mathcal S_{2 \times 2}$ which satisfies condition $\eqref{widehat W}$. Again consider $ \dfrac{n}{2}$ number of $2 \times 2$ matrix valued functions $W_1, \dots, W_{\frac{n}{2}}$, where $W_j (\lm) = \widehat W(\lm)$ for each $j= 1,\dots, \frac{n}{2}$. So, we have
	\begin{align*}
	&\lf \pi_{2 [\frac{n}{2}]} \circ \lf W_1, \dots, W_{\frac{n}{2}} \rf \rf'(0) \\
	 &= \Bigg( x_1, \dots, \df{{n \choose \frac{n}{2}-1}}{n} x_1, {n \choose \frac{n}{2}} \df{x_1 + x_{n-1}}{2}, \df{{n \choose \frac{n}{2}-1}}{n} x_{n-1}, \dots    , x_{n-1}, x_n \Bigg) = x 
	\end{align*}
	and 
	\begin{align*}
	\lf \pi_n \circ \lf W_1, \dots, W_{\frac{n}{2}} \rf \rf(0)   = (0,\dots,0).
	\end{align*}
	The function $\vp = \pi_n \circ \lf W_1, \dots, W_{\frac{n}{2}} \rf$ is analytic and it maps $\D$ into $\mathcal K_n^{even} \cap \Gn$. Thus, for any $n \in \mathbb N$ if condition \eqref{condition-yn} holds for $x \in \mathcal K_n$ then there exists an analytic map $\vp : \D \longrightarrow \Gn$ such that $\vp(0)=(0,\dots,0)$ and $\vp'(0)=x$. The proof is now complete.
	
\end{proof}

\noindent \textbf{Remark.} In particular if $n=3$, then the condition \eqref{condition-yn} is necessary and sufficient for the existence of such an interpolating function $\vp$ for any $x \in \C^3$.\\

\noindent The following is a Schwarz type lemma for the symmetrized polydisc.

\begin{thm}\label{Schwarz-lemma-gn-2}
Let $x=(x_1,\dots,x_n)	\in \C^n$. If there exists an analytic map $\vp : \D \longrightarrow \gn$ such that $\vp(0)=(0,\dots,0)$ and $\vp'(0)=x$, then
\[
\max_{1\leq j \leq n-1}\left\{ \df{|x_j|}{\nj}\right\} + |x_n| \leq 1. 
\]	
\end{thm}

\begin{proof}
Follows from Theorem \ref{Schwarz-lemma-Gn-2} as $\gn \subset \Gn$.
\end{proof}

\section{An explicit interpolating function}

\noindent In Theorem $\ref{Schwarz-lemma-Gn-2}$, we proved the existence of an analytic function $\vp$ mapping origin to origin and satisfying $\vp'(0)=(x_1,\dots,x_n)\in \mathcal K_n$. In this section, we show an explicit construction of such a function $\vp$.

\begin{thm}
	Let $x \in \mathcal K_n$. If a function $\vp = (\vp_1,\dots,\vp_n)$ is given by
	\begin{align*}
	\vp(\lm)
	 = \df{\lm}{ 1 + \lm \bar{x}_n r_x} \lf x_1, \dots, {n \choose [\frac{n}{2}]} \df{ x_1}{n}, {n \choose [\frac{n}{2}]} \df{x_{n-1}}{n}, \dots, x_{n-1}, x_n +\lm r_x  \rf
	\end{align*}
	when $n$ is odd, and 
	\begin{align*}
	\vp(\lm)
	& = \df{\lm }{1 + \lm \bar{x}_n r_x} \Bigg( x_1 , \dots, {n \choose \frac{n-2}{2}}  \df{ x_1}{n}, {n \choose \frac{n}{2}} \df{  x_1  +  x_{n-1}}{2 n},\\
	& \qq \qq \qq \qq \qq \qq {n \choose \frac{n-2}{2}} \df{ x_{n-1}}{ n }, \dots,  x_{n-1},  x_n +\lm r_x  \Bigg).
	\end{align*}
	when $n$ is even, where 
	\[
	r_x=\begin{cases}
	0 & \q \text{if }  x_1=0=x_{n-1} \\
	\df{x_1x_{n-1}(n-|x_1|)}{n|x_1|(n-|x_1|-n|x_n|^2)} & \text{if } |x_{n-1}| \leq |x_1| \neq 0 \\
	\\
	\df{x_1x_{n-1}(n-|x_{n-1}|)}{n|x_{n-1}|(n-|x_{n-1}|-n|x_n|^2)} & \text{if } |x_1| \leq |x_{n-1}| \neq 0.
	\end{cases}
	\]
	Then $\vp$ is an analytic map from $\D$ into $\Gn$ with $\vp(0)=(0,\dots,0)$ and $\vp'(0)=x$.
\end{thm}

\begin{proof}

We divide the proof into two cases, $n=3$ and $n>3$ as in Theorem \ref{Schwarz-lemma-Gn-2}. The idea and technique that are used in constructing an interpolating function $\vp$ when $n=3$ are borrowed from Theorem 2.2 in \cite{Young-LMS}.\\

\noindent \textbf{Case-I:} Suppose $n=3$. Then $\mathcal K_3 =\Ccc $. Let $x \in \Ccc$ be such that $\max\Bigg\{ \df{|x_1|}{3}, \df{|x_2|}{3} \Bigg\} + |x_3| \leq 1$. For this particular case, we denote $r_x$ by $l_x$. Then

\begin{equation}\label{formulae-vp}
\vp(\lm)= \df{\lm}{1+\lm \bar{x}_3l_x}(x_1,x_2,l_x\lm + x_3\,),
\end{equation}
where
\begin{equation}\label{formulae-C}
l_x=\begin{cases}
0 & \q \text{if }  x_1=0=x_2 \\
\df{x_1x_2(3-|x_1|)}{3|x_1|(3-|x_1|-3|x_3|^2)} & \text{if } |x_2| \leq |x_1| \neq 0 \\
\\
\df{x_1x_2(3-|x_2|)}{3|x_2|(3-|x_2|-3|x_3|^2)} & \text{if } |x_1| \leq |x_2| \neq 0.
\end{cases}
\end{equation}
	We shall show that the function $\vp$ given by $\eqref{formulae-vp}$ and $\eqref{formulae-C}$ is analytic, $\vp(\D) \subset \G$, $\vp(0)=(0,0,0)$ and $\vp'(0)=x$.
	Suppose $x_1=0=x_2$, then $|x_3| \leq 1$ and $l_x=0$. Consider the function $\vp(\lm)=(0,0,\lm x_3)$, $\lm \in \D$. Clearly $\vp$ is analytic, $\vp(0)=(0,0,0)$ and $\vp'(0)=(0,0,x_3)=x$. Now suppose $|x_2| \leq |x_1| \neq 0$. The function $\vp$, given by \eqref{formulae-vp} and \eqref{formulae-C}, clearly satisfies $\vp(0)=0$. Note that 
	\[
	\vp'(\lm)= \lf \df{x_1}{(1+\lm \bar{x}_3l_x)^2}, \df{x_2}{(1+\lm \bar{x}_3l_x)^2}, \df{x_3 + \lm l_x (2 + \lm \bar x_3 l_x)}{(1+\lm \bar{x}_3l_x)^2} \rf.
	\]
	Hence $\vp'(0)=x$. 
	We shall show that $\vp$ is analytic and $\vp(\D) \subset \G$. 	\\	
	\noindent Consider $Z$ as in equation $\eqref{Z-zeta}$ with $\sigma = \sqrt{1-\df{|x_1|}{3}}$. Then 
	\[
	D_Z= \begin{bmatrix}
	1 & 0 \\
	0 & \sqrt{\dfrac{|x_1|}{3}}
	\end{bmatrix}	
	\text{ and }	 
	D_{Z^*} 
	= \begin{bmatrix}
	\sqrt{\dfrac{|x_1|}{3}} & 0 \\
	0 & 1
	\end{bmatrix}.
	\]
Also consider $B_{\rho}$ as in equation $\eqref{Y-zeta}$, where $\rho$ is given by $\eqref{xi}$. Then $\n B_{\rho} \n \leq 1$ (as we observe in the proof of Theorem \ref{Schwarz-lemma-Gn-2}). Consider the function $W(\lm)= \M_{-Z}(\lm B_{\rho})$, $\lm \in \D$. Then $W \in \mathcal S_{2 \times 2}$. For a contraction $Z$, we know that $Z^* D_{Z^*}= D_Z Z^*$ and hence $D_Z^{-1} Z^* = Z^* D_{Z^*}^{-1}$ when $\|Z\|<1$. Here $D_Z = (1-Z^* Z)^{\frac{1}{2}}$. Thus 
	\begingroup
	\allowdisplaybreaks
	\begin{align*}
	W(\lm)
	&= Z + D_{z^*}\lm B_{\rho}(1+Z^*\lm B_{\rho})^{-1}D_Z \\
	&= Z + \lm(D_{Z^*}B_{\rho}D_Z)(1+\lm Z^*D_{Z^*}^{-1} B_{\rho} D_Z)^{-1}.
	\end{align*}
	\endgroup
	Note that,
	\begingroup
	\allowdisplaybreaks 
	\begin{align*}
	D_{Z^*}B_{\rho} D_Z
	= \begin{bmatrix}
	\df{x_1}{3} & \df{\rho|x_1|}{3} \\
	\\
	\df{-x_3}{\sigma} & \df{x_2}{3}
	\end{bmatrix}
	\q \text{and} \q
	D_{Z^*}^{-1} B_{\rho} D_Z
	=\begin{bmatrix}
	\df{x_1}{|x_1|} & \rho \\
	\\
	\df{-x_3}{\sigma} & \df{x_2}{3}
	\end{bmatrix}.
	\end{align*}
	\endgroup
	So,
	$$(1+\lm Z^*D_{Z^*}^{-1} B_{\rho} D_Z)^{-1} 
	= \df{1}{1+\lm \rho \sigma}
	\begin{bmatrix}
	1+\lm \rho \sigma & 0 \\
	\\
	\df{-x_1\lm \sigma}{|x_1|} & 1
	\end{bmatrix}.$$
	Thus
	\begingroup
	\allowdisplaybreaks
	\begin{align}\label{F-lm}
	 W(\lm) 
	=\begin{bmatrix}
	0 & \sigma \\
	0 & 0
	\end{bmatrix}
	+ \df{\lm}{1+\lm \rho \sigma}
	\begin{bmatrix}
	\df{x_1}{3} & \df{\rho|x_1|}{3} \\
	\\
	v & \df{x_2}{3}
	\end{bmatrix},
	\end{align}
	\endgroup
	where 
	\begingroup
	\allowdisplaybreaks
	\begin{align}\label{w-y}
	 v = v(\lm)  = -\df{x_3}{\sigma}(1 + \lm \rho \sigma) - \df{\lm \sigma x_1x_2}{3|x_1|} 
	=-\df{x_3}{\sigma} -\lm \sigma l_x.  
	\end{align}
	\endgroup
	From $\eqref{F-lm}$ we have 
	\begingroup
	\allowdisplaybreaks
	\begin{align*}
	\det W(\lm) 
	= \df{\lm^2 (x_1x_2 - 3v\rho |x_1|)}{9(1+\lm \rho \sigma)^2} - \df{\lm \sigma v}{(1+\lm \rho \sigma)}.
	\end{align*}
	\endgroup
	Again 
	$$ \rho \sigma = \df{x_1x_2\bar{x}_3(3 - |x_1|)}{3|x_1|(3 - |x_1| -3|x_3|^2)} = \bar{x}_3l_x \; , \q \df{\rho}{\sigma} = \df{x_1x_2\bar{x}_3}{|x_1|(3 - |x_1| -3|x_3|^2)} $$
	and thus
	\begingroup
	\allowdisplaybreaks
	\begin{align*}
	x_1x_2 - 3v\rho |x_1| 
	&=  x_1x_2 + 3 x_3 \df{x_1x_2\bar{x}_3 |x_1|}{|x_1|(3 - |x_1| -3|x_3|^2)} + 3 \lm \bar x_3 (l_x)^2 |x_1|\\
	&= 3|x_1|l_x(1+\lm \bar{x}_3l_x).
	\end{align*}
	\endgroup
	Then 
	\begingroup
	\allowdisplaybreaks
	\begin{align*}
	\det W(\lm) 
	 = \df{3 \lm^2 |x_1|l_x(1+\lm \bar{x}_3l_x)}{9(1 + \lm \bar{x}_3l_x)^2} - \df{\lm \sigma v}{(1 + \lm \bar{x}_3l_x)}
	 = \df{\lm}{1 + \lm \bar{x}_3l_x} (x_3 +\lm l_x).
	\end{align*}
	\endgroup
	From the equation $\eqref{F-lm}$, we have 
	\[
	[W(\lm)]_{11} = \df{\lm x_1}{3 \lf 1 + \lm \bar{x}_3l_x \rf} \q \text{and} \q [W(\lm)]_{22} = \df{\lm x_2}{ 3 \lf 1 + \lm \bar{x}_3l_x \rf }.
	\]
	Consider the function 
	$$\vp (\lm) = \pi \circ W (\lm) = \lf [3W(\lm)]_{11}, 3[W(\lm)]_{22}, \det W(\lm) \rf , \quad \lm \in \D.$$ 
	Then 
	\[
	\vp(\lm) = \lf \df{\lm x_1}{1 + \lm \bar{x}_3l_x}, \df{\lm x_2}{1 + \lm \bar{x}_3l_x} , \df{\lm}{1 + \lm \bar{x}_3l_x} (x_3 +\lm l_x)\rf ,
	\]
	which is of the form given in $\eqref{formulae-vp}$. Since $W \in \mathcal S_{2 \times 2}$, by Theorem \ref{char G 3} the map $\vp$ is an analytic and $\vp(\D) \subset \G$. The case when $|x_1| \leq |x_2| \neq 0$, can be dealt in a similar way. Hence we are done for $n=3$. \\
		
\noindent \textbf{Case-II:} Suppose $n> 3$. Let $x = (x_1, \dots, x_n) \in \mathcal K_n$. First suppose $n$ is odd. Then,  $x_j = \dfrac{\nj}{n} x_1$ and $x_{n-j} = \dfrac{\nj}{n} x_{n-1}$ for all $j = 2, \dots, \Big[ \dfrac{n}{2} \Big]$. Hence,
	\begin{equation*}
	\max_{1\leq j \leq n-1}\left\{ \df{|x_j|}{\nj}\right\} + |x_n| \leq 1 \q \Longleftrightarrow \q \max\Bigg\{ \df{|x_1|}{n}, \df{|x_{n-1}|}{n} \Bigg\} + |x_n| \leq 1.
	\end{equation*}
Consider $y = (y_1, y_2,y_3) = \lf \dfrac{3 x_1}{n}, \dfrac{3 x_{n-1}}{n}, x_n \rf$.
Therefore, by hypothesis
	\begin{equation*}
	\max\Bigg\{ \df{|y_1|}{3}, \df{|y_2|}{3} \Bigg\} + |y_3| \leq 1.
	\end{equation*}
 So, by Case-I, there exists $\widehat W \in \mathcal S_{2 \times 2} $ such that 
 \[
 [\widehat W(\lm)]_{11} = \df{\lm y_1}{3 \lf 1 + \lm \bar{y}_3l_y \rf}, \q [\widehat W(\lm)]_{22} = \df{\lm y_2}{ 3 \lf 1 + \lm \bar{y}_3l_y \rf } 
 \]
and $ \q \det \widehat W(\lm) = \df{\lm}{1 + \lm \bar{y}_3l_y} (y_3 +\lm l_y), \q$ where 
\[
	l_y=\begin{cases}
	0 & \q \text{if }  y_1=0=y_2 \\
	\df{y_1y_2(3-|y_1|)}{3|y_1|(3-|y_1|-3|y_3|^2)} & \text{if } |y_2| \leq |y_1| \neq 0 \\
	\\
	\df{y_1y_2(3-|y_2|)}{3|y_2|(3-|y_2|-3|y_3|^2)} & \text{if } |y_1| \leq |y_2| \neq 0.
	\end{cases}
\]
Substituting the values of $y$, we get the following:
\begin{align}
&(i) \q l_y = r_x  \notag\\
& (ii) \q \det \widehat W(\lm) = \df{\lm}{1 + \lm \bar{x}_n r_x} (x_n +\lm r_x)  \label{W 2} \\
& (iii) \q
[\widehat W(\lm)]_{11} = \df{\lm x_1}{n \lf 1 + \lm \bar{x}_n r_x \rf}\,, \q [\widehat W(\lm)]_{22} = \df{\lm x_{n-1}}{ n \lf 1 + \lm \bar{x}_n r_x \rf }.\label{W 3}
\end{align}
Consider the $2 \times 2$ matrix valued functions $W_1, \dots, W_{[\frac{n}{2}]}$, where $W_j (\lm) = \widehat W(\lm)$ for each $j= 1,\dots, \lt \frac{n}{2} \rt$. Then each $W_j \in \mathcal S_{2 \times 2}$. Since $n$ is odd, $2\lt \frac{n}{2} \rt +1 = n$. Therefore,
\begin{align*}
&\lf \pi_{2\lt \frac{n}{2} \rt +1} \circ \lf W_1, \dots, W_{[\frac{n}{2}]} \rf \rf(\lm)\\
& = \df{\lm}{ \lf 1 + \lm \bar{x}_n r_x \rf} \lf  x_1 , \dots, {n \choose [\frac{n}{2}]} \df{ x_1}{n},  {n \choose [\frac{n}{2}]} \df{x_{n-1}}{ n }, \dots,  x_{n-1}, (x_n +\lm r_x)  \rf.
\end{align*}

Now suppose $n$ is even. So, $\lt \frac{n}{2} \rt=\frac{n}{2}$. In this case, $x = (x_1, \dots, x_n) \in \mathcal K_n^{even}$. Thus $x_j = \dfrac{\nj}{n} x_1$, $x_{n-j} = \dfrac{\nj}{n} x_{n-1}$ for all $j = 2, \dots, \frac{n}{2}-1$ and ${x}_{\frac{n}{2}} = { n \choose \frac{n}{2}} \df{x_1 + x_{n-1}}{2n}.$ 
Again,
\begin{align*}
\max_{1\leq j \leq n-1}\left\{ \df{|x_j|}{\nj}\right\} + |x_n| \leq 1 
 \Leftrightarrow \max\Bigg\{ \df{|x_1|}{n}, \df{|x_{n-1}|}{n} \Bigg\} + |x_n| \leq 1.
\end{align*}
Therefore, as in the case when $n$ is odd, there exists $\widehat W \in \mathcal S_{2 \times 2} $ such that \eqref{W 2} and \eqref{W 3} hold. Now consider the $2 \times 2$ matrix valued functions $W_1, \dots, W_{\frac{n}{2}}$, where $W_j (\lm) = \widehat W(\lm)$ for each $j= 1,\dots, \frac{n}{2} $. Then
\begin{align*}
\lf \pi_{2 [\frac{n}{2}]} \circ \lf W_1, \dots, W_{\frac{n}{2}} \rf \rf(\lm) \qq &\\
 = \df{\lm }{ \lf 1 + \lm \bar{x}_n r_x \rf} \Bigg( x_1 , \dots, {n \choose \frac{n-2}{2}} & \df{ x_1}{n}, {n \choose \frac{n}{2}} \df{  x_1  +  x_{n-1}}{2 n},\\
&  {n \choose \frac{n-2}{2}} \df{ x_{n-1}}{ n }, \dots,  x_{n-1},  x_n +\lm r_x  \Bigg).
\end{align*}
In both cases $\vp = \pi_n \circ \lf W_1, \dots, W_{[\frac{n}{2}]} \rf$. Clearly, the function $\vp$ is an analytic map from $\D$ into $\Gn$ and $\vp(0)=(0,\dots,0)$. If $n$ is odd, then
\begin{align*}
\vp'(\lm)
 = \df{1}{ \lf 1 + \lm \bar{x}_n r_x \rf^2} \Bigg(  x_1, \dots,& {n \choose [\frac{n}{2}]} \df{ x_1}{n},{n \choose [\frac{n}{2}]} \df{ x_{n-1}}{ n},\\ 
& \dots,   x_{n-1}, x_n + \lm r_x \lf 2+ \lm \bar{x}_n r_x \rf   \Bigg).
\end{align*}
If $n$ is even, we have
\begin{align*}
\vp'(\lm)
 = \df{1}{ \lf 1 + \lm \bar{x}_n r_x \rf^2} \Bigg(  x_1 &, \dots,  {n \choose \frac{n-2}{2}} \df{ x_1}{n},  {n \choose \frac{n}{2}} \df{ \lf x_1 + x_{n-1} \rf}{2n },\\
& {n \choose \frac{n-2}{2}}  \df{ x_{n-1}}{n}, \dots,   x_{n-1}, x_n + \lm r_x \lf 2 + \lm \bar{x}_n r_x \rf   \Bigg).
\end{align*}
It is evident that in either cases $\vp'(0)=x$ and the proof is complete. 

\end{proof}

\section{Geometric interplay between the members of $\Gn$ and $\widetilde{\Gamma_n}$}

\noindent In \cite{pal-roy 2}, we have witnessed several important geometric properties of $\Gn$ and $\widetilde{\Gamma_n}$, e.g., $\widetilde{\Gamma_n}$ is polynomially convex but not convex, $\Gn$ is starlike but not circled etc. In this section, we shall see some interplay between $\Gn$ (or $\widetilde{\Gamma_n}$) and $\widetilde{\mathbb G}_{n+1}$ (or, $\widetilde{\Gamma_{n+1}}$).

\begin{thm}\label{even 1}
	Let $n \in \mathbb N$. Suppose $y = (y_1,\dots, y_{n-1} , q) \in \C^n$
	\begin{itemize}[leftmargin=*]
    \item[$(1)$]
	If $n$ is an even number, then the point 
	$y \in \Gn \;($or, $\in \Gamn)$ if and only if $\hat y \in \widetilde{\mathbb{G}}_{n+1} \;($or, $\in \widetilde{\Gamma}_{n+1})$, where
	\begin{align*}
	\hat y = 
	\Big( \dfrac{n+1}{n} y_1,\dots, \dfrac{n+1}{n+1-j} & y_j, \dots, \dfrac{n+1}{\frac{n}{2}+1}y_{\frac{n}{2}},  \dfrac{n+1}{\frac{n}{2}+1}y_{\frac{n}{2}},\\ \underbrace{\dfrac{n+1}{\frac{n}{2}+2}y_{\frac{n}{2}+1}}_{(\frac{n}{2}+2)\text{-th position}}, \dots, &\underbrace{\dfrac{n+1}{n+1-j}y_{n-j}}_{n+1-j\text{-th position}}, \dots, \dfrac{n+1}{n} y_{n-1}, q \Big).
	\end{align*}	
	\item[$(2)$]
	If $n$ is an odd number, then the point 
	$y \in \Gn \;($or, $\in \Gamn)$ if and only if $y^{\ast} \in \widetilde{\mathbb{G}}_{n+1} \;($or, $\in \widetilde{\Gamma}_{n+1})$, where
	\begin{align*}
	y^{\ast} = 
	\Big( \dfrac{n+1}{n} y_1,&\dots, \dfrac{n+1}{n+1-j}y_j, \dots, \dfrac{2(n+1)}{n+3}y_{[\frac{n}{2}]}, ( y_{[\frac{n}{2}]} + y_{[\frac{n}{2}]+1} ),\\ 
	&\underbrace{\dfrac{2(n+1)}{n+3}y_{[\frac{n}{2}]+1}}_{([\frac{n}{2}]+2)\text{-th position}}, \dots, \underbrace{\dfrac{n+1}{n+1-j}y_{n-j}}_{n+1-j\text{-th position}}, \dots, \dfrac{n+1}{n} y_{n-1}, q \Big) .
	\end{align*}
\end{itemize}
\end{thm}

\begin{proof}

$(1)$. First note that ${n+1 \choose j} = \dfrac{n+1}{n+1-j} \nj$. As $n$ is even, $\Big[ \dfrac{n}{2} \Big] = \dfrac{n}{2} = \Big[ \dfrac{n+1}{2} \Big]$. Suppose $y \in \Gn$. Then, by Theorem \ref{char G 3}, we have
	\[
	{n \choose j}\left|y_j - \bar y_{n-j} q\right| + \left|y_j y_{n-j} - {n \choose j}^2 q \right| < {n \choose j}^2 -|y_{n-j}|^2 \; \text{ for } j = 1,\dots, \frac{n}{2} \,.
	\]
	Consider the point $\hat y = \lf \hat y_1,\dots,\hat y_n, \hat q \rf$, where $\hat q = q$ and
	\[
	\hat y_j = \frac{n+1}{n+1-j} y_j, \q \hat y_{n+1-j} = \frac{n+1}{n+1-j} y_{n-j} \q \text{for }j = 1, \dots, \frac{n}{2}.
	\]

	Then, the following holds for each $j = 1,\dots, \dfrac{n}{2}$.
	\begingroup
	\allowdisplaybreaks
	\begin{align*}
	 & {n+1 \choose j} \left|\hat y_j - \bar{\hat{y}}_{n+1-j} \hat q\right| + \left|\hat y_j \hat y_{n+1-j} - {n+1 \choose j}^2 \hat q \right|\\
	 & = {n+1 \choose j} \dfrac{n+1}{n+1-j}\left|y_j - \bar y_{n-j} q\right| +  \lf \dfrac{n+1}{n+1-j} \rf^2 \left|y_j y_{n-j} - {n \choose j}^2 q \right|\\
	= & \lf \dfrac{n+1}{n+1-j} \rf^2 \lt {n \choose j}\left|y_j - \bar y_{n-j} q\right| + \left|y_j y_{n-j} - {n \choose j}^2 q \right| \rt \\
	< & \lf \dfrac{n+1}{n+1-j} \rf^2 \lt {n \choose j}^2 -|y_{n-j}|^2 \rt\\
	= & {n+1 \choose j}^2 - |\hat y_{n+1-j}|^2 .
	\end{align*}
	\endgroup
	Therefore, by Theorem \ref{char G 3}, $\hat y \in \widetilde{\mathbb{G}}_{n+1}$.\\

Conversely, suppose $\hat y \in \widetilde{\mathbb{G}}_{n+1}$. Then for each $j = 1,\dots, \Big[ \dfrac{n+1}{2} \Big] $, we have 
   	\[
   	{n+1 \choose j}\left|\hat y_j - \bar{\hat{y}}_{n+1-j} \hat q\right| + \left|\hat y_j \hat y_{n+1-j} - {n+1 \choose j}^2 \hat q \right| < {n+1 \choose j}^2 - |\hat y_{n+1-j}|.
   	\]
   	Similarly, we have
   	\[
   	{n \choose j}\left|y_j - \bar y_{n-j} q\right| + \left|y_j y_{n-j} - {n \choose j}^2 q \right| < {n \choose j}^2 -|y_{n-j}|^2
   	\]
   	for any $j = 1,\dots,\dfrac{n}{2}$. Consequently, by Theorem \ref{char G 3}, $y \in \Gn$. In a similar fashion one can prove that $y \in \Gamn$ if and only if $\hat y \in \widetilde{\Gamma}_{n+1}$.\\ 
	
	\noindent $(2)$. Suppose $n$ is odd and suppose $y  = (y_1,\dots, y_{n-1} , q) \in \Gn$. Then $|q|<1$ and there exists $\lf \beta_1,\dots, \beta_{n-1}\rf \in \C^{n-1}$ such that  
	\[
	y_j = \beta_j + \bar{\beta}_{n-j}q, \q  y_{n-j} = \beta_{n-j} + \bar{\beta}_j q \q \text{and} \q |\beta_j|+ |\beta_{n-j}| < {n \choose j}\,,
	\]
	for each $j = 1,\dots, \lt \dfrac{n}{2}\rt $. 
	Since $n$ is odd, $\lt \dfrac{n}{2}\rt = \dfrac{n-1}{2} $ and $\lt \dfrac{n+1}{2} \rt = \dfrac{n+1}{2}$. Consider $(\gamma_1, \dots, \gamma_n) \in \C^n$, where 
	\[ \gamma_{\frac{n+1}{2}} = \beta_{\frac{n-1}{2}} + \beta_{\frac{n+1}{2}}, \; \; \gamma_j= \dfrac{n+1}{n+1-j} \beta_j \; \text{ and } \; \gamma_{n+1-j}= \dfrac{n+1}{n+1-j} \beta_{n-j}\,, \]
	for $j = 1,\dots, \lt \dfrac{n}{2} \rt$.
	Then, we have
	\[ 
	2 |\gamma_{\frac{n+1}{2}}| \leq 2 \lf |\beta_{[\frac{n}{2}]}| + |\beta_{[\frac{n}{2}]}| \rf < \dfrac{n+1}{n+1 - [\frac{n+1}{2}]} {n \choose [\frac{n}{2}]} = {n+1 \choose [\frac{n+1}{2}]}
	\]
	 and  
	\[
	|\gamma_j|+ |\gamma_{n+1-j}| = \dfrac{n+1}{n+1-j} \lf |\beta_j|+ |\beta_{n-j}| \rf < \dfrac{n+1}{n+1-j} {n \choose j} = {n+1 \choose j}\,,
	\] 
	for all $j = 1,\dots, \lt \dfrac{n}{2} \rt$.
	Therefore,
	\[
	\big( \gamma_1 + \overline{\gamma}_n q, \dots, \underbrace{\gamma_j + \overline{\gamma}_{n+1-j} q}_{j \text{-th position}}, \dots, \gamma_n + \overline{\gamma}_1 q, q \big) \in \widetilde{\mathbb G}_{n+1}.
	\]
	Also 
	\[
	\gamma_{\frac{n+1}{2}} + \overline{\gamma}_{\frac{n+1}{2}} q = \beta_{[\frac{n}{2}]} + \overline{\beta}_{[\frac{n}{2}]+1} q + \beta_{[\frac{n}{2}]+1} + \overline{\beta}_{[\frac{n}{2}]} q = y_{[\frac{n}{2}]} + y_{[\frac{n}{2}]+1}
	\]
	and 
	\begin{align*}
	\gamma_j + \overline{\gamma}_{n+1-j} q &= \dfrac{n+1}{n+1-j} \lf \beta_j + \overline{\beta}_{n-j}q \rf = \dfrac{n+1}{n+1-j} y_j, \\
	\gamma_{n+1-j} + \overline{\gamma}_j q &= \dfrac{n+1}{n+1-j} \lf \beta_{n-j} + \overline{\beta}_jq \rf = \dfrac{n+1}{n+1-j} y_{n-j}\,,
	\end{align*}
	for all $j = 1,\dots, \lt \dfrac{n}{2} \rt$.
	Thus, 
	\begin{align*}
	\Big( \dfrac{n+1}{n} y_1,\dots, \dfrac{n+1}{n+1-j}y_j,& \dots, \dfrac{2(n+1)}{n+3}y_{[\frac{n}{2}]}, ( y_{[\frac{n}{2}]} + y_{[\frac{n}{2}]+1} ),\\ \underbrace{\dfrac{2(n+1)}{n+3}y_{[\frac{n}{2}]+1}}_{([\frac{n}{2}]+2)\text{-th position}}, \dots, &\dfrac{n+1}{n+1-j}y_{n-j}, \dots, \dfrac{n+1}{n} y_{n-1}, q \Big) 
	\in \widetilde{\mathbb{G}}_{n+1} .
	\end{align*}
Conversely, suppose $y^{\ast} = \lf y^{\ast}_1, \dots, y^{\ast}_n, q \rf \in \widetilde{\mathbb{G}}_{n+1}$. Then 
\[
y^{\ast}_{\frac{n+1}{2}} = y_{[\frac{n}{2}]} + y_{[\frac{n}{2}]+1}, \; \; y^{\ast}_j = \dfrac{n+1}{n+1-j}y_j, \; \text{ and } \; y^{\ast}_{n+1-j} = \dfrac{n+1}{n+1-j}y_{n-j}\,,
\]
for all $j = 1, \dots, \dfrac{n-1}{2}$. By definition, there exists $(\gamma_1, \dots, \gamma_n) \in \C^n$ such that 
\[
	y^{\ast}_j = \gamma_j + \bar{\gamma}_{n+1-j}q, \; \;  y^{\ast}_{n+1-j} = \gamma_{n+1-j} + \bar{\gamma}_j q \; \text{ and } \; |\gamma_j|+ |\gamma_{n+1-j}| < {n+1 \choose j}\,,
\]
for each $j = 1,\dots, \lt \dfrac{n+1}{2}\rt $. Consider $(\beta_1, \dots, \beta_{n-1}) \in \C^{n-1}$, where 
	\[ \beta_j= \dfrac{n+1-j}{n+1} \gamma_j , \; \beta_{n-j}= \dfrac{n+1}{n+1-j} \gamma_{n+1-j} \q \text{ for } j = 1,\dots,\frac{n-1}{2}. \]
	Then 
	\[
	|\beta_j| + |\beta_{n-j}| = \dfrac{n+1-j}{n+1} \lf |\gamma_j|+ |\gamma_{n+1-j}| \rf < \dfrac{n+1-j}{n+1} {n+1 \choose j} = \nj \,,
	\]
	for each $j = 1,\dots, \lt \dfrac{n}{2}\rt $. Therefore,
	\[
	\big(\beta_1 + \overline{\beta}_{n-1} q, \dots, \beta_j + \overline{\beta}_{n-j} q, \dots, \beta_{n-1} + \overline{\beta}_1 q, q \big) \in \Gn
	\] 
	Note that,
	\begin{align*}
	\beta_j + \overline{\beta}_{n-j} q & = \dfrac{n+1-j}{n+1} \lf \gamma_j + \bar{\gamma}_{n+1-j}q \rf = \dfrac{n+1-j}{n+1} y^{\ast}_j = y_j, \\
	\beta_{n-j} + \overline{\beta}_j q & = \dfrac{n+1-j}{n+1} \lf \gamma_{n+1-j} + \bar{\gamma}_j q \rf = \dfrac{n+1-j}{n+1} y^{\ast}_{n+1-j} = y_{n-j}
	\end{align*}
	for each $j = 1,\dots, \lt \dfrac{n}{2}\rt $. Therefore, $y = (y_1,\dots, y_{n-1} , q) \in \Gn$. The proof of $y \in \Gamn$ if and only if $\hat y \in \widetilde{\Gamma}_{n+1}$ is similar.
	
\end{proof}

\begin{thm}\label{even 2}
	Let $y = (y_1,\dots, y_{n-1} , q) \in \C^n$.
	\begin{itemize}[leftmargin=*]
	\item[$(1)$]
	If $n$ is even and $y \in \Gn \;($or, $\in \Gamn )$, then $\check y \in \widetilde{\mathbb{G}}_{n-1} \;($or, $\in \widetilde{\Gamma}_{n-1})$, where
	\begin{align*}
		\check y = 
		\Big( \dfrac{n-1}{n} y_1,  \dots, \dfrac{n-j}{n}  y_j, \dots, \dfrac{\frac{n}{2}+1}{n} y_{\frac{n}{2}-1},& \underbrace{\dfrac{\frac{n}{2}+1}{n}y_{\frac{n}{2}+1}}_{\frac{n}{2}\text{-th position}},\\ 
		\dfrac{\frac{n}{2}+2}{n}y_{\frac{n}{2}+2}, \dots, \underbrace{\dfrac{n-j}{n}y_{n-j}}_{(n-1-j)\text{-th position}}, & \dots, \dfrac{n-1}{n} y_{n-1}, q \Big).
	\end{align*}
	\item[$(2)$]
	If $n$ is odd and $y \in \Gn \;($or, $\in \Gamn)$, then $\tilde y \in \widetilde{\mathbb{G}}_{n-1} \;($or, $\in \widetilde{\Gamma}_{n-1})$, where
	\begin{align*}
	\tilde y = 
	\Big( \dfrac{n-1}{n} y_1,  \dots , \dfrac{n-j}{n} y_j &, \dots, \dfrac{n+3}{2n} y_{\frac{n-3}{2}}, \underbrace{\dfrac{n+1}{2n}\lf \dfrac{ y_{\frac{n-1}{2}} + y_{\frac{n+1}{2}}}{2} \rf}_{\frac{n-1}{2}\text{-th position}},\\ 
	\dfrac{n+3}{2n}& y_{\frac{n+3}{2}}, \dots, \underbrace{\dfrac{n-j}{n}y_{n-j}}_{(n-j-1)\text{-th position}},\dots, \dfrac{n-1}{n} y_{n-1}, q \Big).
	\end{align*}
	\end{itemize}
\end{thm}

\begin{proof}
	$(1)$. Let $n\in \mathbb N$ be even. Then, $\lt \frac{n-1}{2} \rt = \frac{n}{2}-1$. Also it is merely mentioned that $\nj = \dfrac{n}{n-j} {n-1 \choose j}$. Now suppose $y \in \Gn$. Then by Theorem \ref{char G 3}, we have for each $j \in \ls 1,\dots,\Big[ \dfrac{n}{2} \Big] \rs$
\[
{n \choose j}\left|y_j - \bar y_{n-j} q\right| + \left|y_j y_{n-j} - {n \choose j}^2 q \right| < {n \choose j}^2 -|y_{n-j}|^2 .
\] 
Consider the point
\begin{align*}
\check y &= \lf \check y_1, \dots, \check y_{n-2}, \check q \rf \\ &=\Big( \dfrac{n-1}{n} y_1,\dots, \dfrac{\frac{n}{2}+1}{n} y_{\frac{n}{2}-1}, \underbrace{\dfrac{\frac{n}{2}+1}{n}y_{\frac{n}{2}+1}}_{\frac{n}{2}\text{-th position}}, \dfrac{\frac{n}{2}+2}{n}y_{\frac{n}{2}+2},\dots, \dfrac{n-1}{n} y_{n-1}, q \Big).
\end{align*}
	Then $\check q = q$ and 
	\[
	\check y_j = \frac{n-j}{n} y_j , \q \hat y_{n-1-j} = \frac{n-j}{n} y_{n-j} \q \text{for } j = 1, \dots, \frac{n}{2}-1 .
	\]
	Note that for each $j = 1, \dots, \dfrac{n}{2}-1$, we have
	\begin{align*}
	&\left|\check y_j - \bar{\check{y}}_{n-1-j} \check q\right| = \dfrac{n-j}{n}\left|y_j - \bar y_{n-j} q\right| \\
	\text{and} \q 
	&\left|\check y_j \check y_{n-1-j} - {n-1 \choose j}^2 \check q \right| =  \lf \dfrac{n-j}{n} \rf^2 \left|y_j y_{n-j} - {n \choose j}^2 q \right|.
	\end{align*}
	Hence, for any $j = 1,\dots,\dfrac{n}{2}-1\; (= [ \frac{n-1}{2} ] ) $, we have
	\begingroup
	\allowdisplaybreaks
	\begin{align*}
		& {n-1 \choose j}\left|\check y_j - \bar{\check{y}}_{n-1-j} \check q\right| + \left|\check y_j \check y_{n-1-j} - {n-1 \choose j}^2 \check q \right|\\
		=\, & \lf \dfrac{n-j}{n} \rf^2 \lt {n \choose j}\left|y_j - \bar y_{n-j} q\right| + \left|y_j y_{n-j} - {n \choose j}^2 q \right| \rt \\ 
		<\, & \lf \dfrac{n-j}{n} \rf^2 \lt {n \choose j}^2 -|y_{n-j}|^2 \rt\\
		=\, & {n-1 \choose j}^2 - |\check y_{n-1-j}|.
	\end{align*}
	\endgroup
	Therefore, by Theorem \ref{char G 3}, we conclude that $\check y \in \widetilde{\mathbb{G}}_{n-1}$. Similarly if $y \in \Gamn$, then $\check y \in \widetilde{\Gamma}_{n-1}$. \\
	
\noindent $(2)$. Suppose $n$ is odd and let $y \in \Gn$. Then $|q|<1$ and there exists $\lf \beta_1,\dots, \beta_{n-1}\rf \in \C^{n-1}$ such that  
	\[
	y_j = \beta_j + \bar{\beta}_{n-j}q, \q  y_{n-j} = \beta_{n-j} + \bar{\beta}_j q \q \text{and} \q |\beta_j|+ |\beta_{n-j}| < {n \choose j}\,,
	\]
	for each $j \in \ls 1,\dots, \lt \dfrac{n}{2}\rt \rs$. Consider the given point $\tilde y = \lf \tilde y_1, \dots \tilde y_{n-2}, \tilde q \rf \in \C^{n-1}$. Then $\tilde q = q$, 
	\[\tilde y_{\frac{n-1}{2}}= \dfrac{n+1}{2n}\lf \dfrac{ y_{\frac{n-1}{2}} + y_{\frac{n+1}{2}}}{2} \rf, \]
	\[ \tilde y_j = \dfrac{n-j}{n}y_j \q \text{and} \q \tilde y_{n-j-1} = \dfrac{n-j}{n}y_{n-j}\,, \q \text{for } j = 1,\dots, \frac{n-3}{2}.\]
    Consider $\lf \gamma_1, \dots, \gamma_{n-2} \rf \in \C^{n-2}$, where  $\gamma_{\frac{n-1}{2}} = \dfrac{n+1}{2n} \lf \dfrac{\beta_{\frac{n-1}{2}} + \beta_{\frac{n+1}{2}} }{2} \rf $,
	\[\gamma_j= \dfrac{n-j}{n} \beta_j \q \text{and} \q \gamma_{n-1-j}= \dfrac{n-j}{n} \beta_{n-j}\,, \q \text{for } j = 1,\dots, \frac{n-3}{2}.\]
	Then, we have
	\[
	2\left| \gamma_{\frac{n-1}{2}} \right| = \dfrac{n+1}{2n} \left|\beta_{\frac{n-1}{2}} + \beta_{\frac{n+1}{2}} \right|
	< \dfrac{n- \frac{n-1}{2}}{n}{n \choose {\frac{n-1}{2}}} ={n-1 \choose \frac{n-1}{2}} 
	\]
	\[\text{and} \q
	|\gamma_j|+ |\gamma_{n-1-j}| = \dfrac{n-j}{n} \lf |\beta_j|+ |\beta_{n-j}| \rf < \dfrac{n-j}{n} {n \choose j} = {n-1 \choose j}\,,
\]
for $j = 1,\dots, \dfrac{n-3}{2}$. Therefore,
\[
\big( \gamma_1 + \overline{\gamma}_{n-2} q, \dots, \underbrace{\gamma_j + \overline{\gamma}_{n+1-j} q}_{j \text{-th position}}, \dots, \gamma_{n-2} + \overline{\gamma}_1 q, q \big) \in \widetilde{\mathbb G}_{n-1}.
\]	
Also we have,
\begin{align*}
\gamma_{\frac{n-1}{2}} + \bar{\gamma}_{\frac{n-1}{2}}  q &= \dfrac{n+1}{2n} \lf \dfrac{\beta_{\frac{n-1}{2}} + \overline{\beta}_{\frac{n+1}{2}} q + \beta_{\frac{n+1}{2}} + \overline{\beta}_{\frac{n-1}{2}} q  }{2} \rf \\ 
&= \dfrac{n+1}{2n} \dfrac{\lf y_{\frac{n-1}{2}} + y_{\frac{n+1}{2}} \rf}{2} = \tilde y_{\frac{n-1}{2}},
\end{align*} 
\begin{align*} 
&  \gamma_j + \bar{\gamma}_{n-1-j}q = \dfrac{n-j}{n} \lf \beta_j + \bar{\beta}_{n-j}q \rf = \dfrac{n-j}{n}y_j = \tilde y_j\\  
\text{and} \q &  \gamma_{n-1-j} + \bar{\gamma}_j q = \dfrac{n-j}{n} \lf \beta_{n-j} + \bar{\beta}_j q \rf = \dfrac{n-j}{n}y_{n-j} = \tilde y_{n-j-1}\,,
\end{align*}	
for $j = 1,\dots, \dfrac{n-3}{2}$. Hence, $|\tilde q|< 1$ and there exists $\lf \gamma_1, \dots, \gamma_{n-2} \rf \in \C^{n-2}$ such that
	\[
	\tilde y_j = \gamma_j + \bar{\gamma}_{n-1-j} \tilde q , \q  \tilde y_{n-j-1} = \gamma_{n-1-j} + \bar{\gamma}_j \tilde q \q \text{and} \q |\gamma_j|+ |\gamma_{n-1-j}| < {n-1 \choose j}\,,
	\]
	for each $j = 1,\dots, \dfrac{n-1}{2} $. Consequently $\tilde y = \lf \tilde y_1, \dots \tilde y_{n-2}, \tilde q \rf \in \widetilde{\mathbb{G}}_{n-1}$. The proof of $y \in \Gamn$ implies $\tilde y \in \widetilde{\Gamma}_{n-1}$ is similar.
\end{proof}

\begin{thm}\label{even 3}
	Let $n \in \mathbb N$ be even.
	\item[$(1)$]
	Let the point $y = (y_1,\dots, y_{n-1} , q) \in \Gn \;($or, $\in \Gamn)$. Then the point 
	\[ \underline y = \big( y_1,\dots, y_{\frac{n}{2}}, y_{\frac{n}{2}},\underbrace{ y_{\frac{n}{2} + 1}}_{(\frac{n}{2} + 2)-th}, \dots, y_{n-1}, q \big) \in \widetilde{\mathbb G}_{n+1} \;(\text{or}, \in \widetilde{\Gamma}_{n+1} ) \]
	and the map $f : \Gn \rightarrow \widetilde{\mathbb G}_{n+1}$ that maps $y$ to $\underline y$ is an analytic embedding.\\
	\item[$(2)$]
	Let $y = (y_1,\dots, y_n , q) \in \widetilde{\mathbb G}_{n+1}\;($or, $\in \widetilde{\Gamma}_{n+1})$. Then the point $\tilde y \in \Gn \;($or, $\in \Gamn)$, where
	\begin{align*}
		\tilde y = \Bigg( \frac{n}{n+1}y_1,  \dots, & \frac{n+1-j}{n+1} y_j, \dots, \frac{\frac{n}{2} +2}{n+1} y_{\frac{n}{2} - 1}, \frac{\frac{n}{2}+1}{2(n+1)}\lf y_{\frac{n}{2}}+ y_{\frac{n}{2} +1} \rf,\\
		&\underbrace{\frac{\frac{n}{2} +2}{n+1} y_{\frac{n}{2} +2}}_{(\frac{n}{2} + 1)-\text{th}}, \dots, \frac{n+1-j}{n+1} y_{n+1-j}, \dots, \frac{n}{n+1}y_n, q \Bigg).
	\end{align*}
	The map $g : \widetilde{\mathbb G}_{n+1} \rightarrow \Gn$ that maps $y$ to $\tilde y$ is analytic.\\
	\item[$(3)$]
	Let $y = (y_1,\dots, y_n , q) \in \widetilde{\mathbb G}_{n+1}\;($or, $\in \widetilde{\Gamma}_{n+1})$. Then the point 
	\[ \hat y = \lf \dfrac{y_1}{{n+1 \choose 1}}, \dots, \dfrac{y_j}{{n+1 \choose j}}, \dots, \dfrac{y_n}{{n+1 \choose 1}},q \rf \in  \widetilde{\mathbb G}_{n+1}  \;(\text{or}, \in \widetilde{\Gamma}_{n+1} )\]
	and the point 
	\[
	\hat y_{\sharp} = \Big( \hat y_1, \dots, \hat y_{\frac{n}{2}-1}, \underbrace{\dfrac{\hat y_{\frac{n}{2}} + \hat y_{\frac{n}{2}+1} }{2}}_{\frac{n}{2}-\text{th}}, \hat y_{\frac{n}{2} +2}, \dots, \hat y_n, q \Big) \in \Gn \;(\text{or, } \in \Gamn) \,,
	\]
where $\hat y_j = \dfrac{y_j}{{n+1 \choose j}}$ for each $j$. Also
the function $h : \widetilde{\mathbb G}_{n+1} \rightarrow \Gn$ that maps $y$ to $\hat y_{\sharp}$ is analytic.
\end{thm}
\begin{proof}
	
	\item[$(1)$].
	Since $y = (y_1,\dots, y_{n-1} , q) \in \Gn$, there exists a unique $\lf \beta_1, \dots, \beta_{n-1} \rf \in \C^{n-1}$ such that 
	\[y_j = \beta_j + \bar{\beta}_{n-j} q, \q y_{n-j} = \beta_{n-j} + \bar{\beta}_j q \q \text{and} \q |\beta_j| + |\beta_{n-j}|< \nj \,, \]
	for each $j = 1, \dots, \frac{n}{2}$. Note that $ y_{\frac{n}{2}} = \beta_{\frac{n}{2}} + \bar{\beta}_{\frac{n}{2}} q $.	
	Consider the given point $\underline y = \lf \underline y_1, \dots, \underline y_n, \underline q \rf \in \C^{n+1}$. Then $\underline q =q$,
	\[
	\underline y_j = y_j \q \text{and} \q 	\underline y_{n+1-j} = y_{n-j} \q \text{for } j = 1, \dots, \frac{n}{2}.
	\]
	Define $ \gamma_j = \beta_j$ and $\gamma_{n+1-j} = \beta_{n-j}$ for $j = 1,\dots, \frac{n}{2}$. Then,
	\[
	\gamma_1 = \beta_1, \dots, \gamma_{\frac{n}{2}}= \beta_{\frac{n}{2}},\; \gamma_{\frac{n}{2}+1}= \beta_{\frac{n}{2}}, \; \gamma_{\frac{n}{2}+2}= \beta_{\frac{n}{2}+1}, \dots, \gamma_n = \beta_{n-1}.
	\]
	Evidently,
	\[|\gamma_j|+ |\gamma_{n+1-j}| = |\beta_j| + |\beta_{n-j}| < \nj \leq {n+1 \choose j}\,, \q \text{for }\; j = 1,\dots, \frac{n}{2}.\] 
	For $j = 1, \dots, \frac{n}{2}$, we have
	\[\gamma_j + \bar{\gamma}_{n+1-j} \underline q = y_j =\underline y_j   \q \text{and} \q \gamma_{n+1-j} + \bar{\gamma}_j \underline q = y_{n-j} = \underline y_{n+1-j}   \;\; . \]
	Hence, there exists $\lf \gamma_1, \dots, \gamma_n \rf \in \C^n$ such that 
	\[\underline y_j = \gamma_j + \bar{\gamma}_{n+1-j} q, \; \; \underline y_{n+1-j} = \gamma_{n+1-j} + \bar{\gamma}_j q \; \text{ and } \; |\gamma_j| + |\gamma_{n+1-j}|< {n+1 \choose j} ,\]
	for each $j = 1, \dots, \lt \frac{n+1}{2} \rt$.
	Hence $y \in \widetilde{\mathbb G}_{n+1} $. The map $f$ is clearly an analytic embedding of $\Gn$ into $\widetilde{\mathbb G}_{n+1}$.\\
	
	\item[$(2)$]			
	Let $y = (y_1,\dots, y_n , q) \in \widetilde{\mathbb G}_{n+1}$. Then, there exists $\lf \beta_1, \dots, \beta_n \rf \in \C^n$ such that 
	\[
	y_j = \beta_j + \bar{\beta}_{n+1-j} q, \;\;  y_{n+1-j} = \beta_{n+1-j} + \bar{\beta}_j q \; \text{ and }\; |\beta_j| + |\beta_{n+1-j}|< {n+1 \choose j}\,,
	\]
	for $j = 1, \dots, n$. Consider the point 
	$ \tilde y = \lf \tilde y_1, \dots,\tilde y_{n-1}, \tilde q \rf \in \C^n$, where $\tilde q =q$,
	$\tilde y_{\frac{n}{2}}=  \dfrac{\frac{n}{2}+1}{2(n+1)}  \lf y_{\frac{n}{2}} + y_{\frac{n}{2} + 1} \rf$  and
	\[\tilde y_j = \dfrac{n+1-j}{n+1} y_j, \;\;
	\tilde y_{n-j} = \dfrac{n+1-j}{n+1} y_{n+1-j}\, \;\; \text{ for } j = 1, \dots, \frac{n}{2}-1.\]	
	Define $\lf \gamma_1, \dots, \gamma_{n-1}\rf \in \C^{n-1}$ as follows  
	\[ \gamma_j = \dfrac{n+1-j}{n+1} \beta_j, \q  \gamma_{n-j}= \dfrac{n+1-j}{n+1} \beta_{n+1-j}\, \q \text{for } j = 1,\dots, \frac{n}{2}-1 \]
	\[ \text{and} \q \gamma_{\frac{n}{2}}= \dfrac{n+1- \frac{n}{2}}{2(n+1)}\lf \beta_{\frac{n}{2}} + \beta_{\frac{n}{2}+1} \rf = \dfrac{n+2}{4(n+1)}\lf \beta_{\frac{n}{2}} + \beta_{\frac{n}{2}+1} \rf. \]
	Then for $j = 1,\dots, \dfrac{n}{2}-1$, we have 
	\[ 
	|\gamma_j| + |\gamma_{n-j}| = \dfrac{n+1-j}{n+1} \lf|\beta_j| + |\beta_{n+1-j}| \rf < \dfrac{n+1-j}{n+1} {n+1 \choose j} = \nj 
	\]
	\[
	\text{and }\q 2|\gamma_{\frac{n}{2}}| \leq \dfrac{n+2}{2(n+1)}\lf |\beta_{\frac{n}{2}}| + |\beta_{\frac{n}{2}+1}| \rf < \dfrac{n+1- \frac{n}{2}}{(n+1)} {n+1 \choose \frac{n}{2}} = {n \choose \frac{n}{2}} .
	\]
	Therefore, $\lf \gamma_1, \dots, \gamma_{n-1}\rf \in \C^{n-1}$ where $|\gamma_j| + |\gamma_{n-j}| < \nj $, for all $j = 1, \dots, \frac{n}{2}$. Also
	\begin{align*}
		\tilde y_j &=  \dfrac{n+1-j}{n+1} y_j =  \dfrac{n+1-j}{n+1} \lf \beta_j + \bar{\beta}_{n+1-j} q \rf = \gamma_j + \bar{\gamma}_{n-j} \tilde q, \\
		\tilde y_{n-j} &=  \dfrac{n+1-j}{n+1}y_{n+1-j} =  \dfrac{n+1-j}{n+1} \lf \beta_{n+1-j} + \bar{\beta}_j q \rf = \gamma_{n-j} + \bar{\gamma}_j \tilde q \,,
	\end{align*}
	for $j = 1, \dots, \frac{n}{2}-1 $. Since $y_{\frac{n}{2}} = \beta_{\frac{n}{2}} + \bar{\beta}_{\frac{n}{2}+1} q \;$ and  $ \; y_{\frac{n}{2}+1} = \beta_{\frac{n}{2}+1} + \bar{\beta}_{\frac{n}{2}} q$, we have that
	\begin{align*}
		\tilde y_{\frac{n}{2}} &= \dfrac{n+2}{4(n+1)} \lf y_{\frac{n}{2}} + y_{\frac{n}{2} + 1} \rf \\
		&= \dfrac{n+2}{4(n+1)} \lt \lf \beta_{\frac{n}{2}} + \beta_{\frac{n}{2}+1} \rf + \lf \bar{\beta}_{\frac{n}{2}} + \bar{\beta}_{\frac{n}{2}+1} \rf q \rt \\
		&= \gamma_{\frac{n}{2}} + \bar{\gamma}_{\frac{n}{2}} \tilde q.
	\end{align*}
	Thus, there exists $\lf \gamma_1, \dots, \gamma_{n-1}\rf \in \C^{n-1}$ such that 
	\[\tilde y_j = \gamma_j + \bar{\gamma}_{n-j} \tilde q , \q \tilde y_{n-j} = \gamma_{n-j} + \bar{\gamma}_j \tilde q \q \text{and} \q |\gamma_j| + |\gamma_{n-j}| < \nj \,,
	\] 
	for all $j = 1, \dots,\frac{n}{2} $. Hence $\tilde y \in \Gn$. The map $g$ is obviously analytic.\\
	
	\item[$(3)$]
	Let $y = (y_1,\dots, y_n , q) \in \widetilde{\mathbb G}_{n+1}$. Then there exists $\lf \beta_1, \dots, \beta_n \rf \in \C^n$ such that 
	\[
	y_j = \beta_j + \bar{\beta}_{n+1-j} q, \;\;  y_{n+1-j} = \beta_{n+1-j} + \bar{\beta}_j q \; \text{ and } |\beta_j| + |\beta_{n+1-j}|< {n+1 \choose j}\,,
	\]
	for $j = 1, \dots,n$. Consider the point $ \hat y = \lf \hat y_1, \dots, \hat y_n, \hat q \rf $, where $\hat q= q$, and\\
	$\hat y_j = \dfrac{y_j}{{n+1 \choose j}}$ for all $j = 1, \dots, n$. Let $\alpha_j = \dfrac{\beta_j}{{n+1 \choose j}}$ for $j = 1, \dots, n$. Then 
	\[|\alpha_j| + |\alpha_{n+1-j}| = \dfrac{|\beta_j| + |\beta_{n+1-j}|}{{n+1 \choose j}} < 1 \leq {n+1 \choose j}.\] 
	Also we have
	\[
	\hat y_j = \dfrac{y_j}{{n+1 \choose j}} = \alpha_j + \bar{\alpha}_{n+1-j} q, \q  \hat y_{n+1-j}= \dfrac{y_{n+1-j}}{{n+1 \choose j}}= \alpha_{n+1-j} + \bar{\alpha}_j q \,,
	\]
	for $j = 1, \dots, n$. Thus, $\hat y \in  \widetilde{\mathbb G}_{n+1}$.	Next define  $\lf \gamma_1, \dots, \gamma_{n-1}\rf \in \C^{n-1}$ in the following way:
	\[\gamma_{\frac{n}{2}}= \dfrac{\lf \alpha_{\frac{n}{2}} + \alpha_{\frac{n}{2}+1} \rf}{2}, \q \gamma_j = \alpha_j \q \text{and} \q \gamma_{n-j}= \alpha_{n+1-j} \]
	for $j = 1,\dots, \frac{n}{2}-1$.
	Then, we have \[
	2|\gamma_{\frac{n}{2}}| \leq  |\beta_{\frac{n}{2}}| + |\beta_{\frac{n}{2}+1}|  < 1 \leq {n \choose \frac{n}{2}} 
	\]
	\[ 
	\text{and }\q |\gamma_j| + |\gamma_{n-j}| = |\alpha_j| + |\alpha_{n+1-j}|  < 1 \leq \nj \q \text{for } j = 1,\dots, \frac{n}{2}-1.
	\]		
	Therefore, $\lf \gamma_1, \dots, \gamma_{n-1}\rf \in \C^{n-1}$ where $|\gamma_j| + |\gamma_{n-j}| < \nj $, for all $j = 1, \dots, \frac{n}{2}$. Note that,
	\begin{align*}
		\hat y_j & = \gamma_j + \bar{\gamma}_{n-j} q \\
		\hat y_{n-j} & = \gamma_{n-j} + \bar{\gamma}_j q \q \text{for } j = 1, \dots, \frac{n}{2}-1 .
	\end{align*}
	Since $\hat y_{\frac{n}{2}} = \alpha_{\frac{n}{2}} + \bar{\alpha}_{\frac{n}{2}+1} q \;$ and  $ \;\hat y_{\frac{n}{2}+1} = \alpha_{\frac{n}{2}+1} + \bar{\alpha}_{\frac{n}{2}} q$, we also have
	\begin{align*}
		\dfrac{\hat y_{\frac{n}{2}} + \hat y_{\frac{n}{2} + 1}}{2} 
		&= \dfrac{\lf \alpha_{\frac{n}{2}} + \alpha_{\frac{n}{2}+1} \rf}{2} + \dfrac{\lf \bar{\alpha}_{\frac{n}{2}} + \bar{\alpha}_{\frac{n}{2}+1} \rf}{2} q \\
		&= \gamma_{\frac{n}{2}} + \bar{\gamma}_{\frac{n}{2}} q.
	\end{align*}
	Hence, the point
	\[
	\hat y_{\sharp} = \lf \hat y_1, \dots, \hat y_{\frac{n}{2}-1}, \dfrac{\hat y_{\frac{n}{2}} + \hat y_{\frac{n}{2}+1} }{2}, \hat y_{\frac{n}{2} +2}, \dots, \hat y_n, q \rf \in \Gn .
	\]
	Clearly, the map $h$ is analytic.
\end{proof}

\begin{thm}\label{odd 3}
	Let $n \in \mathbb N$ be odd.
	\item[$(1)$]
	Let the point $y = (y_1,\dots, y_{n-1} , q) \in \Gn \,($or, $\in \Gamn)$. Then the point 
	\[ \underline y = \Big( y_1,\dots, y_{[\frac{n}{2}]}, \df{y_{[\frac{n}{2}]}+ y_{[\frac{n}{2}] + 1}}{2},\underbrace{ y_{\frac{n}{2} + 1}}_{(\frac{n}{2} + 2)-th}, \dots, y_{n-1}, q \Big) \in \widetilde{\mathbb G}_{n+1} \; (\text{or}, \in \widetilde{\Gamma}_{n+1} ) \]
	and the map $f : \Gn \rightarrow \widetilde{\mathbb G}_{n+1}$ that maps $y$ to $\underline y$ is analytic.\\
	\item[$(2)$]
	Let $y = (y_1,\dots, y_n , q) \in \widetilde{\mathbb G}_{n+1}\;($or, $\in \widetilde{\Gamma}_{n+1})$. Then the point $\tilde y \in \Gn \;($or, $\in \Gamn)$, where
	\begin{align*}
	\tilde y = \Bigg( \frac{n}{n+1}y_1, & \dots,  \frac{n+1-j}{n+1} y_j, \dots, 
	  \frac{n+3}{2(n+1)}y_{[\frac{n}{2}]},\\
	&\underbrace{\frac{n+3}{2(n+1)} y_{[\frac{n}{2}] +2}}_{([\frac{n}{2}] + 1)-\text{th}}, \dots, \frac{n+1-j}{n+1} y_{n+1-j}, \dots, \frac{n}{n+1}y_n, q \Bigg).
	\end{align*}
	The map $g : \widetilde{\mathbb G}_{n+1} \rightarrow \Gn$ that maps $y$ to $\tilde y$ is an analytic embedding.\\
	\item[$(3)$]
	Let $y = (y_1,\dots, y_n , q) \in \widetilde{\mathbb G}_{n+1}\;($or, $\in \widetilde{\Gamma}_{n+1})$. Then the point 
	\[ \hat y = \lf \dfrac{y_1}{{n+1 \choose 1}}, \dots, \dfrac{y_j}{{n+1 \choose j}}, \dots, \dfrac{y_n}{{n+1 \choose 1}},q \rf \in  \widetilde{\mathbb G}_{n+1} \;(\text{or}, \in \widetilde{\Gamma}_{n+1} )\]
	and the point 
	\[
	\hat y_{\sharp} = \Big( \hat y_1, \dots, \hat y_{[\frac{n}{2}]-1}, \underbrace{\hat y_{[\frac{n}{2}]} }_{\frac{n}{2}-\text{th}}, \hat y_{[\frac{n}{2}] +2}, \dots, \hat y_n, q \Big) \in \Gn(\text{or, } \in \Gamn)\,,
	\]
where $\hat y_j = \dfrac{y_j}{{n+1 \choose j}}$. The map $h : \widetilde{\mathbb G}_{n+1} \rightarrow \Gn$ that maps $y$ to $\hat y_{\sharp}$ is an analytic embedding.
\end{thm}
\begin{proof}
	The proof is similar to that of Theorem \ref{even 3}.
\end{proof}

\end{document}